\documentclass{amsart}
\usepackage[hyperfootnotes=false]{hyperref}
\usepackage{fullpage}
\usepackage{amsrefs}
\usepackage{tikz}
\usepackage{dsfont}
\usepackage{newverbs}
\usepackage{fancyvrb}
\usepackage{comment}
\usepackage{mathrsfs}
\usepackage{mathtools}
\usepackage{amssymb}
\usepackage{stmaryrd}

\usepackage{amsthm}

\usepackage{easytable, pdflscape, enumitem, euscript, amsmath}
\usepackage[graphicx]{realboxes}

\usepackage{bbm}
\usepackage{array}

\usepackage{float}

\usepackage{tikz}
\usetikzlibrary{matrix,arrows,decorations.pathmorphing, cd}

\usepackage{longtable}

\newif\ifshowkeys
\showkeysfalse

\ifshowkeys

\else

\fi

\newcommand{\LA}	{{\mathbb{A}}}
\newcommand{\LB}	{{\mathbb{B}}}
\newcommand{\LC}	{{\mathbb{C}}}

\newcommand{\LT}        {{\mathbb{T}}}

\newcommand{\LL}        {{\mathbb{L}}}
\newcommand{\LQ}        {{\mathbb{Q}}}

\newcommand{\CC}        {{\mathcal{C}}}
\newcommand{\CD}        {{\mathcal{D}}}
\newcommand{\CE}        {{\mathcal{E}}}
\newcommand{\CF}        {{\mathcal{F}}}

\newcommand{\CN}        {{\mathcal{N}}}

\newcommand{\CP}        {{\mathcal{P}}}
\newcommand{\CS}        {{\mathcal{S}}}
\newcommand{\CT}        {{\mathcal{T}}}

\newcommand{\fp}        {{\mathfrak{p}}}
\newcommand{\fq}        {{\mathfrak{q}}}
\newcommand{\ft}        {{\mathfrak{t}}}
\newcommand{\fm}        {{\mathfrak{m}}}

\newcommand{\mbu}         {\mathbbm{1}}

\newcommand{\bc}[1]	{\langle #1\rangle}

\newcommand{\da}       {\text{-}}

\newcommand{\holim}  {\operatornamewithlimits{\underset{\longleftarrow}{holim}}}

\newtheorem{theorem}{Theorem}[section]
\newtheorem{conjecture}[theorem]{Conjecture}
\newtheorem*{conjecture*}{Conjecture}
\newtheorem{lemma}[theorem]{Lemma}
\newtheorem{proposition}[theorem]{Proposition}
\newtheorem{corollary}[theorem]{Corollary}
\theoremstyle{definition}

\newtheorem{remark}[theorem]{Remark}
\newtheorem{definition}[theorem]{Definition}
\newtheorem{example}[theorem]{Example}

\numberwithin{equation}{theorem}

\theoremstyle{theorem}
\newtheorem*{theorem*}{Theorem}

\DefineVerbatimEnvironment{checks}{Verbatim}{formatcom=\color{blue}}
\DefineVerbatimEnvironment{mcodeblock}{Verbatim}{formatcom=\color{darkred}}
\definecolor{darkred}{rgb}{0.5,0,0}
\definecolor{darkgreen}{rgb}{0,0.5,0}
\newverbcommand{\mcode}{\color{darkred}}{}
\newverbcommand{\fname}{\color{darkgreen}}{}

\newcommand\blfootnote[1]{%
  \begingroup
  \renewcommand\thefootnote{}\footnote{#1}%
  \addtocounter{footnote}{-1}%
  \endgroup
}

\title{A conjecture on the composition of localizations on a stratified tensor triangulated category}
\author{Nicola~Bellumat}

\begin{document}
\blfootnote{This research was supported by grant no.~DNRF156 from the Danish National Research Foundation and by CEMS.UL - Center for Mathematical Studies, University of Lisbon, FCT Portugal, UID/04561/2025.} 
\maketitle

\begin{abstract}
We study the composition of Bousfield localizations on a tensor triangulated category stratified via the Balmer-Favi support and with noetherian Balmer spectrum. Our aim is to provide reductions via purely axiomatic arguments, allowing us general applications to concrete categories examined in mathematical practice. We propose a conjecture which states that the behaviour of the composition of the localizations depends on the chains of inclusions of the Balmer primes indexing said localizations. We prove this conjecture in the case of finite or low dimensional Balmer spectra.    
\end{abstract}

\section{Introduction}
Our starting point is \cite{itloc}: in this paper we considered a collection of Bousfield classes labelled by a finite totally ordered set and showed that under reasonable assumptions all the possible compositions of localizations with respect to joins of the examined Bousfield classes are well-behaved. We provided a diagram displaying all these compositions and relating them by canonical natural transformations arising from the universal properties of the localizations. This diagram is indexed over a finite poset and we can define a binary operation on it which keeps track of the composition of the localizations. Such a construction allows us great control on the localizations and we can reduce the computations of their compositions to combinatorics of sets.

The main application was chromatic homotopy theory: our preferred category was $\CS p$, the stable homotopy category, and the Bousfield classes we had in mind were the ones individuated by the Morava $K$-theories $K(i)$ for $i\leq n$, since we had to impose an upper bound to the chromatic height. However, the final proof could be reduced to purely categorical arguments without any computation relying on the specific properties of these cohomology theories.

In the present work, we examine what happens if, instead of a finite totally ordered set, the Bousfield classes are indexed over a possibly non-linear poset. The first question is to find natural occurrences of this situation, so to have available concrete examples and guarantee interesting applications. This is why the theory of stratification is introduced: it correlates the classification of the localizing tensor ideals of a tensor triangulated category to a space encapsulating the geometric properties of this category  by means of a suitable notion of ``support". Such a classification provides both a meaningful set of Bousfield classes and an indexing set with some structure condensing their intrinsic relation to one another.  

The exact details of the theory of stratification vary depending on the authors and sources considered. One of its first appearances was in the context of modular representation theory using support varieties, thanks to the work of Quillen, Carlson, Rickard, and others. Other crucial results were developed by Benson, Iyengar, and Krause, who defined a notion of support on a compactly generated tensor triangulated category admitting an action of a graded noetherian ring (see \cite{bik08} and \cite{bik}). The approach we will follow in this work is the one recently provided by Barthel, Heard and Sanders in \cite{BHS-stratification}: they use the notion of support developed by Balmer and Favi in \cite{balmerfavi}, which extends the Balmer support previously defined (in \cite{balmer-spectrumprimes}) only for the compact objects of a compactly generated category. This version not only allows us to recover the previous chromatic example by considering $L_n \CS p$, the localization of the stable homotopy category with respect to the $n$-th Johnson-Wilson theory, it also provides an elegant formulation of the notion of stratification which does not rely on the action of a group or a graded ring on the category in question.

Therefore, our basic set-up will be the following: let $(\CT, \otimes, \mbu)$ be a rigid-compactly generated tensor triangulated category with weakly noetherian Balmer spectrum $\text{Spc}(\CT^c)$ and stratified via the Balmer-Favi support. For any Balmer prime $\fp\in \text{Spc}(\CT^c)$ we can form an idempotent $g(\fp)$ representing the support at $\fp$. For a general subset $A\subseteq \text{Spc}(\CT^c)$ we define $\LL_A$ to be the Bousfield localization with respect to the coproduct of the objects $g(\fp)$ for $\fp\in A$, i.e.\ $\LL_A$ is the localization whose kernel is the localizing subcategory $\{t\in \CT : \coprod_{\fp \in A}g(\fp)\otimes t=0\}$. Given a $k$-tuple of subsets $\LA=(A_1,\dots, A_k)$ we define $\LL_{\LA}$ to be the composition $\LL_{A_1}\LL_{A_2}\dots \LL_{A_k}$. These iterated localizations will be the principal object of our study. 

The crucial point becomes what is the relation between these localizations, their compositions and the topology of the Balmer spectrum: while  the subspace structure of $A\subseteq \text{Spc}(\CT^c)$ reveals much information about the single localization $\LL_{A}$, it is much more difficult to deduce the properties of $\LL_{\LA}$ from the tuple $\LA$. The composition of two (or more) localizations is not necessarily again a localization. Hence, we cannot expect to classify it by means of its kernel, thus relating it to a subset of the Balmer spectrum by stratification. Even if the stratification gives a complete classification of the localizing tensor ideals of the category, it does not provide a direct description of the endofunctors obtained by composing the localizations with respect to these localizing ideals. The monoid generated by such endofunctors under composition presents quite a complicated behaviour. Nevertheless, it is worth understanding it since iterated localizations appear in multiple and various contexts, representing gluing data between objects. For a famous example, consider the Chromatic Splitting Conjecture of Hopkins.

The following question arises: are the compositions of the localizations determined uniquely by the Balmer spectrum? We propose a conjecture which tries to provide a positive answer to this question, by giving a sufficient condition for two iterated localizations $\LL_{\LA}$ and $\LL_{\LB}$ being isomorphic based on the combinatorics of the tuples $\LA$ and $\LB$.

We will need the following notions. If $\LA=(A_1,\dots, A_k)$ is a $k$-tuple of subsets of $\text{Spc}(\CT^c)$, then a \textit{thread} of $\LA$ is a sequence of inclusions of Balmer primes $\fp_1\supseteq \fp_2\supseteq \dots \supseteq \fp_k$ such that $\fp_i\in A_i$. A \textit{thread set} for $\LA$ is a chain $T\subseteq \text{Spc}(\CT^c)$ such that $(T\cap A_1,T\cap A_2, \dots, T\cap A_k)$ admits a thread, i.e. there exists a sequence of inclusions $\fp_1\supseteq \dots \supseteq \fp_k$ with $\fp_i\in T\cap A_i$.

Finally we can state our conjecture
\begin{conjecture*}
Let $\CT$ be a rigid-compactly generated tensor triangulated category. Suppose it is stratified in the sense of \cite[Def.~4.4]{BHS-stratification}, its Balmer spectrum is noetherian and of finite Krull dimension.

Let $\LA$ and $\LB$ be two tuples of subsets of $\emph{Spc}(\CT^c)$. Suppose these tuples have the same thread sets. Then, there is a canonical isomorphism between the associated iterated localizations $\LL_{\LA}\cong \LL_{\LB}$.
\end{conjecture*}

This Conjecture is a direct generalization of the result \cite[Prop.~1.10]{itloc}, where we considered a finite linear order. In the case of \cite{itloc} the fracture cube argument allowed us to decompose a localization $\LL_A$ as homotopy limit of finitely many localizations $\LL_{A'}$ for subsets $A'\subset A$ satisfying the right conditions. In the current context, where we may want to consider infinite posets, this argument is not enough to encase all the possible iterated localizations in a diagram relating them as in \cite[Thm.~1.8]{itloc}. However, even if this strategy seems not to be easily adapted, all the results we could prove and all the examples we could compute seem to confirm the Conjecture. Indeed, we could prove the Conjecture in the following situations
\begin{theorem*}
Assume the tensor triangulated category $\CT$ admits a model. Then, the Conjecture holds in the following cases:
\begin{itemize}
\item[(a)] if $\emph{Spc}(\CT^c)$ is finite;
\item[(b)] if $\emph{Spc}(\CT^c)$ has dimension 1;
\item[(c)] if $\emph{Spc}(\CT^c)$ has dimension 2 and has finitely many minimal primes.
\end{itemize}
\end{theorem*}
The category $\CT$ admitting a model is a technical condition for our arguments to go through. It ensures that we have well-behaved homotopy limits. See Remark~\ref{rmk-categorywithmodel-sectionreduction} for the exact details.

The structure of the paper will be the following: in Section~\ref{section-stratification} we first recall some of the basic notions and facts of the theory of Balmer spectra and the approach to stratification developed by Barthel, Heard and Sanders in \cite{BHS-stratification}. Finally, in Section~\ref{section-basics} we use these notions to establish our set-up, present the notation we will adopt for localizations and their composition and then spell out the above Conjecture as Conjecture~\ref{conj}. In Section~\ref{section-firstreductions} we provide reductions that allow us to trim the tuples labelling iterated localizations by removing elements from their subsets which do not fit in descending chains of inclusions of Balmer primes. In Section~\ref{section-finitecase} we adapt the argument of \cite{itloc} to obtain analogous results for the case where we limit ourselves to a finite subset of the Balmer spectrum $\text{Spc}(\CT^c)$. The proof of Conjecture~\ref{conj} when the Balmer spectrum itself is finite follows automatically (Corollary~\ref{cor-conjturefinitecase-sectionfinitecase}). We conclude with Section~\ref{section-lowdimcases} where we prove Conjecture~\ref{conj} for Balmer spectra of low dimensions (Theorem~\ref{thm-conjtruedim1-sectionlowdim} and Theorem~\ref{thm-conjtruedim2irred-sectionlowdim}).\\

\textbf{Acknowledgements:} I would like to thank Neil Strickland for giving me access to his personal notes and sharing his insight on various matters, in particular Example~\ref{ex-BPcomputationthread-sectionbasics} was indicated by him. Many thanks also to John Greenlees, who suggested trying to adapt my previous work to settings different from chromatic homotopy theory and whose computations formed the basis for Example~\ref{ex-SpcO(2)counterexamplefractureaxiom-sectionstratification}, and to Greg Stevenson, for his invaluable help with the theory of Balmer spectra.
\section{Stratification via Balmer-Favi support}
\label{section-stratification}
As mentioned in the Introduction, our main reference for the theory of tensor triangulated categories is \cite{BHS-stratification} and we will be adopting most of its notation. However, we establish our own notation regarding localization functors.

We recall that the Balmer spectrum $\text{Spc}(\CT^c)$ associated with a rigidly-compactly generated tensor triangulated category $\CT$ is a \textit{spectral space} in the sense of Hochster. For the definition and the properties of this class of topological spaces we take as reference the monograph \cite{spespa}.

\begin{definition}
Let $(P,\leq)$ be a partially ordered set. We denote by $\CP(P)$ its power set and by $\CP(P)'=\CP(P)\setminus \{\emptyset\}$ the collection of non-empty subsets. Let $S\subseteq P$ be an arbitrary subset, then we define the following collections
\begin{align*}
[\leq S]&=\{p \in P: \exists s  \in S \ p \leq s\} \\
[\geq S]&=\{p \in P: \exists s  \in S \ p \geq s\} \\
[\not \leq S]&=[\leq S]^c=\{p\in P : \forall s\in S \ p\not \leq s\}\\
[\not \geq S]&=[\geq S]^c=\{p \in P: \forall s  \in S \ p \not \geq s\}.  
\end{align*}
These, in the language of posets, are the family and cofamily generated by the set $S$ and their complements. In the case where $S=\{s\}$ is a singleton we will simply write $[\leq s], [\geq s]$ and $[\not \leq s], [\not \geq s]$.

Our main application will be for the Balmer spectrum $\text{Spc}(\mathcal{T}^c)$ with the ordering given by the inclusions of Balmer primes.

However, in some cases we will provide a homeomorphism of topological spaces $\text{Spc}(\mathcal{T}^c)\cong P$ which is an isomorphism of posets $(\text{Spc}(\mathcal{T}^c), \subseteq )\cong (P, \leq)$ for some partial ordering on $P$. In this situation we will identify the family and the cofamily generated by some $S\subseteq \text{Spc}(\mathcal{T}^c)$ with the respective family and cofamily in $P$.
\end{definition}

\begin{remark}\label{rmk-thomasondownwardclosed-sectionbalmer}
From \cite[Prop.~2.9]{balmer-spectrumprimes} it is clear that any Thomason subset $Y$ of a Balmer spectrum $\text{Spc}(\CT^c)$ is closed under specialization, or downward closed. That is, if we have two Balmer primes $\fp, \fq$ such that $\fp \in Y$ and $\fq \subseteq \fp$, then it follows $\fq \in Y$.
\end{remark}

\begin{definition}\label{def-noetherian-sectionbalmer}
A topological space is \textit{noetherian} if any descending chain of closed subsets admits a minimal element. This is equivalent to any open subset of the space being quasi-compact. 

\cite[Prop.~2.14]{balmer-spectrumprimes} establishes that an open subset of $\text{Spc}(\mathcal{T}^c)$ is quasi-compact if and only if it is the complement of the support of an object of $\mathcal{T}^c$. Therefore, a Balmer spectrum is noetherian if and only if any closed subspace is the support of an object. This condition is equivalent to any prime $\mathfrak{p}$ being \textit{visible}, i.e. there exists $X \in \mathcal{T}^c$ such that $\text{supp}(X)=\overline{\{\mathfrak{p}\}}$. See \cite[Cor.~7.14]{balmerfavi}.
\end{definition}
This is an important property because it reduces the topology of the Balmer spectrum to the closure of its points, which is determined by the poset structure given by the inclusion of prime ideals.

More explicitly, in noetherian Balmer spectra Thomason subsets can be completely characterized by the inclusion ordering.
\begin{lemma}\label{lem-thomason=downwardclosed-sectionbalmer}
Let $\CT$ be rigidly-compactly generated tensor triangulated category and suppose $\emph{Spc}(\CT^c)$ is noetherian. Then, a subset $Y\subseteq \emph{Spc}(\CT^c)$ is Thomason if and only if it is downward closed.
\end{lemma}

\begin{example}\label{ex-notnoetherianspectalspace-sectionbalmer}
For a spectral space $X$ we say a subset $Y\subseteq X$ is specialization closed if $y\in Y$ implies $\overline{\{y\}}\subseteq Y$. In this generality, a spectral space $X$ is noetherian if and only if its Thomason subsets coincide with the specialization closed subsets, see \cite[Prop.~7.13]{balmerfavi}.

For example, let us consider the space $\big\{\frac{1}{n} : n\geq 1\big\}\cup \{0\}$ as subspace of the real numbers. It can be verified directly that this is a spectral space which is not noetherian. It is T1, hence all its subsets are specialization closed. But clearly there are subsets which are not Thomason: e.g. $\{0\}$ is a closed subset whose complement is not compact.
\end{example}

%A priori the spectrum $\text{Spc}(\CT^c)$ does not provide a classification of localizing tensor ideals of the whole tensor triangulated category $\mathcal{T}$. Indeed, as we saw the Balmer spectrum is actually defined for the compact objects $\CT^c$ of the tensor triangulated category: the support of a non-compact object is not defined, thus a priori we cannot expect the localizing subcategories to be catalogued by the Balmer notion of support.
%
%However, with further assumptions we can refine the original approach of Balmer to obtain a classification analogous to \cite[Thm.~4.10]{balmer-spectrumprimes} for localizing ideals.
%
%We first need to provide a notion of support valid also for non-compact objects.
\begin{definition}\label{def-kappaprime-sectionbalmer}
Let $\CT$ be a rigidly-compactly generated tensor triangulated category. Let $Y\subseteq \text{Spc}(\CT^c)$ be a Thomason subset, so we have the corresponding thick tensor ideal of compact objects $\CT^c_{Y}$ by \cite[Thm.~4.10]{balmer-spectrumprimes}. Then we can invoke \cite[Thm.~3.3.5]{hopast:ash} to obtain from $\CT^c_{Y}$ the exact triangle
\[ \Gamma_{Y}\mathbbm{1}\rightarrow \mathbbm{1} \rightarrow L_{Y^c}\mathbbm{1}\rightarrow \Sigma \Gamma_Y\mathbbm{1} \]
where $\Gamma_{Y}\mathbbm{1} \in \text{Loc}_{\otimes}\bc{\CT^c_Y}$ and $L_{Y^c}\mathbbm{1} \in \text{Loc}_{\otimes}\bc{\CT^c_Y}^{\perp}$ are the idempotents associated respectively to the acyclicization and localization with respect to $\text{Loc}_{\otimes}\bc{\CT^c_Y}$.

In the language of \cite[Rmk.~1.23]{BHS-stratification}, we have $\Gamma_Y\mbu=e_Y$ and $L_{Y^c}\mbu=f_Y$. Our choice is motivated by the fact that we want to keep in mind that these two idempotents represent a colocalization and a localization functor respectively. 
\end{definition}

\begin{remark}\label{rmk-smahsinglocinducedinclusionbalmer-sectionbalmer}
In the above set-up, the smashing localization 
\[L_{Y^c}\colon \CT \rightarrow L_{Y^c}\CT \qquad X\mapsto L_{Y^c}\mathbbm{1}\otimes X \]
is finite by construction. Therefore, if we consider its restriction to the compact objects $L_{Y^c}\colon \CT^c\rightarrow L_{Y^c}\CT^c$, \cite[Rmk.~1.23]{BHS-stratification} allows us to identify the function $\text{Spc}(L_{Y^c})$ with the subspace inclusion
\[ \text{Spc}(L_{Y^c}\CT^c)\cong Y^c\hookrightarrow \text{Spc}(\CT^c). \]
\end{remark}
We observe that a localization functor on a tensor triangulated category $\CT$ is uniquely determined by its kernel, which is a localizing tensor ideal of $\CT$. The classification of localizing tensor ideals is the motivating concept of stratification.

Given $\CT$ a rigidly-compactly generated tensor triangulated category with weakly noetherian Balmer spectrum, we will consider the notion of stratification via Balmer-Favi support introduced in  \cite[Def.~4.4]{BHS-stratification}.

Before concluding this section, we make the discussion less abstract by recalling three types of tensor triangulated categories which are known to be stratified. These will provide concrete examples which we will examine later when we study the composition of localizations.

\begin{example}\label{ex-DRnoethstratified-sectionbalmer}
Let $R$ be a commutative ring. We consider $D(R)$ the derived category of $R$-modules. Then, there exists a homeomorphism
\begin{align*}
\text{Spec}(R)&\cong \text{Spc}(D(R)^c) \\
P&\mapsto \mathfrak{p}=\{ M \in D(R)^c : M_P\simeq 0 \}
\end{align*} 
where $\text{Spec}(R)$ is the classical Zariski spectrum and $M_P$ denotes the localization at the prime ideal $P$. This is a consequence of the more general classification of prime tensor ideals of the derived category of a quasi-compact quasi-separated scheme provided by Thomason \cite[Thm.~4.1]{thomason-classification}.

We warn the reader that this homeomorphism is inclusion reversing: an inclusion of algebraic primes $P\subseteq Q$ corresponds to the inclusion of Balmer primes $\mathfrak{q}\subseteq \fp$.

If we assume the ring $R$ to be noetherian, then this Balmer spectrum is a noetherian space. In this case, it is a classical result of Neeman that the localizing tensor ideals of $D(R)$ are in bijection with the subsets of $\text{Spec}(R)$ (\cite[Thm.~2.8]{neeman-tower}). In \cite[Ex.~5.7]{BHS-stratification} the relation between Neeman's notion of support and the Balmer-Favi support is discussed and \cite[Thm.~5.8]{BHS-stratification} provides a direct proof that $D(R)$ is stratified in the sense of \cite[Def.~4.4]{BHS-stratification}.
\end{example}

\begin{example}\label{ex-LnSpstratified-sectionbalmer}
We fix $p$ a prime number and $n$ a natural number. We denote by $E(n)$ the $n$-th Johnson-Wilson theory at the prime $p$ and by $K(i)$ the $i$-th Morava K-theory for an arbitrary number $i$. We set $L_n \CS p_{(p)}$ to be the localization with respect to $E(n)$ of the $p$-local stable homotopy category. 

It was proven in \cite[Thm.~6.9]{host-loc} that $\text{Spc}(L_n \CS p_{(p)}^c)$ consists of the prime ideals
\[ \CE_0 \supset \CE_1\supset \dots \supset \CE_{n-1}\supset \CE_n=0   \]
where $\CE_i=\{X\in L_n\CS p_{(p)}^c : K(i)_*X=0\}$. We notice that this Balmer spectrum is trivially noetherian since it is finite.

Furthermore, \cite[Thm.~6.14]{host-loc} establishes a bijection between the localizing subcategories of $L_n \CS p_{(p)}$ and the subsets of $\text{Spc}(L_n \CS p_{(p)}^c)$ using a notion of support defined via the $K(i)$-homology (see \cite[Def.~6.7]{host-loc}).

In \cite[\S 10]{BHS-stratification} it is provided an identification $g(\CE_i)\cong M_{i}S^0$ (the $i$-th monochromatic sphere) and then proved that the Balmer-Favi support coincides with the above notion of support. The conclusion is that $L_n \CS p_{(p)}$ is stratified via Balmer-Favi support.
\end{example}

\begin{example}\label{ex-GSpQstratified-sectionbalmer}
Let $G$ be a compact Lie group. We set $G \da \CS p_{\LQ}$ to be the rational $G$-equivariant stable homotopy category. It is proved in \cite[Thm.~8.4]{gre-rationalbalmer} that the elements of $\text{Spc}(G \da \CS p_{\LQ}^c)$ are the tensor ideals in the form 
\[ {\fp}_H=\{ X \in G \da \CS p_{\mathbb{Q}}^c : \phi^H X \cong 0 \} \] 
where $H$ ranges over the conjugacy classes of closed subgroups $H \leq G$ and $\phi^H$ denotes the classical geometric $H$-fixed points functor. Moreover, we have an inclusion of Balmer primes $\fp_K\subseteq \fp_H$ if and only if $K$ is conjugate to a subgroup cotorally included in $H$ (we write this as $K \leq_{ct} H$), see \cite[Cor.~5.4]{gre-rationalbalmer} and \cite[Cor.~7.4]{gre-rationalbalmer}.

We denote by $(\Gamma G, zf)$ the set of conjugacy classes of closed subgroups of $G$ with the $zf$-topology (defined in \cite[\S 10.B]{gre-rationalbalmer}). Then \cite[Thm.~10.2]{gre-rationalbalmer} establishes a homeomorphism of spaces
\begin{align*}
(\Gamma G, zf)&\rightarrow \text{Spc}(G \da \CS p_{\LQ}^c)\\
[H]&\mapsto \fp_H
\end{align*}
which translates the cotoral inclusion $\leq_{ct}$ to the inclusion of Balmer primes.

We observe that in general the Balmer spectrum $\text{Spc}(G \da \CS p_{\LQ}^c)$ is not noetherian. However, if $G$ is the product of a torus by a finite group then we have noetherianity.

In \cite[\S 12]{BHS-stratification} the results of \cite{gre-rationalbalmer} are interpreted using the Balmer-Favi support to prove that $G \da \CS p_{\LQ}$ is stratified in the sense of \cite[Def.~4.4]{BHS-stratification}.
\end{example}

\section{Bousfield localization on stratified categories} \label{section-basics}
By the very definition of stratification, for stratified tensor triangulated categories the localizing tensor ideals are completely classified by the Balmer-Favi support. Therefore, we can propose the following notation covering all possible Bousfield localizations.
\begin{definition}\label{def-localizationLA-sectionbasics}
Let $\CT$ be a rigidly-compactly generated tensor triangulated category with a weakly noetherian Balmer spectrum which is stratified. Then, for any subset $A\subseteq \text{Spc}(\CT^c)$ we set $\LL_{A}$ to be the localization with respect to the localizing subcategory 
\[ \CT_{A^c}=\{ X \in \CT : \text{Supp}(X)\subseteq A^c \}=\{ X \in \CT : \forall \fp \in A \quad  g(\fp)\otimes X=0 \}. \]
That is, the kernel of the localization $\LL_A$ is given by the subcategory $\CT_{A^c}$.

We define $g(A)=\coprod_{\fp \in A} g(\fp)$. Then,  $\LL_A$ coincides with the Bousfield localization associated with this object. We consider it as an endofunctor of the category $\CT$. If $A$ consists in a singleton, say $\{\fp\}$, we simply write $\LL_{\fp}$ instead of $\LL_{\{\fp\}}$.

Let $\LA=(A_1,\dots, A_k)$ be a $k$-tuple of subsets of the Balmer spectrum $A_i\subseteq \text{Spc}(\CT^c)$, then we set $\LL_{\LA}$ to be the composition of localizations $\LL_{A_1}\LL_{A_2}\dots \LL_{A_k}$.
\end{definition}

\begin{remark}
The assumption that $\CT$ is stratified means that all localizing tensor ideals are set-generated and they are completely classified by the Balmer-Favi support. See \cite[Prop.~3.5]{BHS-stratification} and \cite[Thm.~4.1]{BHS-stratification}.  It follows that any localizing tensor ideal of $\CT$ is in the form $\CT_{A^c}$ for some subset $A\subseteq \text{Spc}(\CT^c)$. I.e., it is the Bousfield class of some object of $\CT$.

Therefore, all localizations on $\CT$ are Bousfield localizations and they are in the form $\LL_A$ as in Definition~\ref{def-localizationLA-sectionbasics}.
\end{remark}

From now on, whenever we invoke a tensor triangulated category this will be a rigidly-compactly generated tensor triangulated category with weakly noetherian Balmer spectrum which is stratified.
\begin{definition}
Given a tensor triangulated category $\CT$ and two objects $X,Y\in \CT$ we will be denoting by $[X,Y]$ the hom-group of morphisms from $X$ to $Y$, i.e. $[X,Y]=\text{Hom}_{\CT}(X,Y)$.

Moreover, we denote by 
\[F(-,-)\colon \CT^{op}\times \CT\rightarrow \CT \]
the internal hom functor. That is, $F(X,-)$ is the exact functor right adjoint to $X\otimes -$. We will write $DX=F(X,\mbu)$ for the dual of the object $X$.
\end{definition}

%\begin{remark}\label{rmk-suppX=suppDX-sectionBousfieldloc}
%Let $\fp \in \text{Spc}(\CT^c)$, then it follows from \cite[Lemma~A.2.6]{hopast:ash} that a compact object $X$ belongs to $\fp$ if and only if $DX \in \fp$. Consequently, $\text{supp}(X)=\text{supp}(DX)$. 
%\end{remark}

\begin{example}
Consider a Thomason subset $A\subseteq \text{Spc}(\CT^c)$. As we explained in Definition~\ref{def-kappaprime-sectionbalmer}, associated with this subset we have an exact triangle
\[\Gamma_A\mathbbm{1}\rightarrow \mathbbm{1}\rightarrow L_{A^c}\mathbbm{1} \]
giving rise to the finite localization 
\[L_{A^c}\colon \CT\rightarrow \CT \qquad X\mapsto X\otimes L_{A^c}\mbu. \]
In virtue of \cite[Lemma~2.13]{BHS-stratification}, the right idempotent $L_{A^c}\mbu$  and the object $g(A^c)$ have the same Bousfield class. Therefore, the finite localization $L_{A^c}$ coincides with $\LL_{A^c}$. It is immediate that for any two Thomason subsets $A, B$ we have $L_{A^c}L_{B^c}\cong L_{A^c\cap B^c}$.

If $\text{Spc}(\CT^c)$ is generically noetherian then \cite[Thm.~9.11]{BHS-stratification} implies that all the possible smashing localizations are of this form and they are actually finite. 
%It is immediate that for any two Thomason subsets $A, B$ we have $L_{A^c}L_{B^c}\cong L_{A^c\cap B^c}$.

Moreover, applying $F(-,X)$ to the starting exact triangle we obtain another triangle
\[ F(L_{A^c}\mbu, X)\rightarrow X\rightarrow F(\Gamma_A\mbu, X). \]
The functor 
\[\Lambda_A\colon \CT\rightarrow \CT \qquad X\mapsto F(\Gamma_A\mbu, X)\]
coincides with the Bousfield localization $\LL_{A}$ (\cite[Thm.~3.3.5]{hopast:ash}). For two Thomason subsets $A, B$ the composition of localizations $\Lambda_A$ and $\Lambda_B$ can be easily computed $\Lambda_A\Lambda_B\cong \Lambda_{A\cap B}$.
\end{example}

\begin{proposition}\label{prop-LambdaALBc=convexloc}
Let $A, B\subseteq \emph{Spc}(\CT^c)$ be two Thomason subsets. Then, the functor
\[\CT\rightarrow \CT \qquad X\mapsto F(\Gamma_A\mbu, L_{B^c}X) \]
is the Bousfield localization with respect to $g(A\cap B^c)$, i.e.\ we have a natural isomorphism $\LL_{(A,B^c)}\cong\LL_{A\cap B^c}$.
\end{proposition}
\begin{proof}
We show directly that for any $X\in \CT$ the morphism $X\rightarrow \Lambda_AL_{B^c}X$ has $g(A\cap B^c)$-local target and its cofiber is $g(A\cap B^c)$-acyclic.

Let $Z$ be any $g(A\cap B^c)$-acyclic spectrum, then we have
\[ [Z, \Lambda_AL_{B^c}X]=[Z, F(\Gamma_A\mbu, L_{B^c}X)]\cong[\Gamma_A \mbu \otimes Z, L_{B^c}X]\cong [L_{B^c} \mbu \otimes \Gamma_{A}\mbu \otimes Z, L_{B^c}X]. \]
By \cite[Lemma~2.13]{BHS-stratification}  we have $\text{Supp}(L_{B^c} \mbu \otimes \Gamma_{A}\mbu \otimes Z)=\text{Supp}(Z)\cap A\cap B^c$ and the acyclic condition on $Z$ means this set is empty. Stratification implies that objects with empty support are trivial, thus the above hom-group is zero.

Since the proposed morphism is given by the composition 
\[ X\rightarrow L_{B^c}X\rightarrow \Lambda_AL_{B^c}X \]
we can use the octahedral axiom to produce an exact triangle of the form
\[ \Sigma\Gamma_{B}X\rightarrow C \rightarrow \Sigma F(L_{A^c}\mbu, L_{B^c}X) \]
where $C$ is the cofiber of the morphism $X\rightarrow \Lambda_AL_{B^c}X$.

If $\fp \in A\cap B^c$ we immediately have $g(\fp)\otimes \Gamma_{B}X=0$, thus we are left with showing $F(L_{A^c}\mbu, L_{B^c}X)$ is $g(A\cap B^c)$-acyclic. We start by observing that there exists a subset $Z\subseteq \text{Spc}(\CT^c)$ such that it is Thomason closed, $\fp \in Z$ and $Z\subseteq A$. It is enough to recall that the inverse topology on $\text{Spc}(\CT^c)$ has as open subsets the Thomason subsets and has an open basis consisting of closed Thomason subsets. See \cite[Def.~1.4.1]{spespa} and \cite[Thm.~1.4.3]{spespa} for a complete reference. Thus, we can form $Z$ by just taking a basis open neighbourhood of $\fp \in A$ in the inverse topology. Since $A$ is open with respect to the inverse topology, this open neighbourhood can be taken to be included in $A$. By \cite[Prop.~2.14~(b)]{balmer-spectrumprimes}, we must have $Z=\text{supp}(\sigma_{\fp})$ for some $\sigma_{\fp} \in \CT^c$. By \cite[Lemma~2.18]{BHS-stratification} it is enough to show $\sigma_{\fp} \otimes F(L_{A^c}\mbu, L_{B^c}X)=0$ to prove that $\fp$ does not belong to the support of $F(L_{A^c}\mbu, L_{B^c}X)$. By compactness we have
\[\sigma_{\fp} \otimes F(L_{A^c}\mbu, L_{B^c}X)\cong F(D\sigma_{\fp} \otimes L_{A^c}\mbu, L_{B^c}X) \]
and $\text{supp}(D\sigma_{\fp})=\text{supp}(\sigma_{\fp})$ holds by \cite[Prop.~2.7~(i)]{balmer-suppfiltr}. By construction $\text{supp}(D\sigma_{\fp})\subseteq A$, therefore $\text{Supp}(D\sigma_{\fp} \otimes L_{A^c}\mbu)=\emptyset$. Thus, the cofiber $C$ is $g(A\cap B^c)$-acyclic as we wanted.
\end{proof}
%\begin{corollary}
%Suppose the Balmer spectrum $\emph{Spc}(\CT^c)$ is noetherian. Then for any two Thomason subsets $A,B\subseteq \emph{Spc}(\CT^c)$ we have $\Lambda_AL_{B^c}\cong \LL_{A\cap B^c}$.
%\end{corollary}

We now introduce the crucial relationship between the composition of Bousfield localizations and the topology of the Balmer spectrum.
\begin{proposition}\label{prop-fractureaxiombalmerspectrum}
Let $\CT$ be as above. Let $A\subseteq \emph{Spc}(\CT^c)$ be an arbitrary subset. Let $\fp$ be a visible Balmer prime such that $\forall \fq \in A$ we have $\fp \not \supseteq \fq$. Then, $\LL_{\fp}\LL_{A}=0$ holds.
\end{proposition}
\begin{proof}
This is just an adaptation of the argument of \cite[Prop.~2.4]{itloc} to the context of Balmer spectra. However, we repeat the proof here for ease of the reader.

Visibility of $\fp$ means the existence of a compact object $\sigma_{\fp}$ whose support coincides with the closed set $\overline{\{\fp\}}$. The claim is equivalent to $g(\fp)\otimes \LL_{A}X=0$ for an arbitrary object $X\in \CT$. But \cite[Lemma~2.18]{BHS-stratification} implies that $\fp$ does not belong to the support of $\LL_{A}X$ if and only if it does not belong to the support of $\sigma_{\fp}\otimes \LL_{A}X$. Therefore, it suffices to show $\sigma_{\fp}\otimes \LL_{A}X=0$.

This is equivalent to showing the identity map of $\sigma_{\fp}\otimes \LL_{A}X$ is trivial. By adjunction this is the same as the corresponding morphism
\[D\sigma_{\fp}\otimes \sigma_{\fp}\otimes \LL_{A}X\rightarrow \LL_{A}X \]
being trivial. But the support of the source is a subset of $\text{supp}(\sigma_{\fp})=\{\fq \in \text{Spc}(\CT^c) : \fq\subseteq \fp\}$ and by assumption this is disjoint from $A$. Hence, the source is $g(A)$-acyclic while the target is $g(A)$-local. It follows the morphism must be trivial as claimed. 
\end{proof}
\begin{corollary}\label{cor-facturenoetherianbalmerspectrum-sectionstratification}
Suppose the Balmer spectrum $\emph{Spc}(\CT^c)$ is noetherian. Let $A,B\subseteq \emph{Spc}(\CT^c)$  be two subsets such that $\forall \fq\in A, \fp\in B$ we have $\fp\not \supseteq \fq$. Then, $\LL_{B}\LL_{A}=0$ holds.
\end{corollary}
\begin{proof}
Proposition~\ref{prop-fractureaxiombalmerspectrum} and the noetherianity assumption imply that $\LL_{\fp}\LL_A=0$ for any prime $\fp\in B$ . This immediately implies that $\LL_B\LL_A=0$.
\end{proof}
One could wonder if the visibility of the Balmer primes is necessary to ensure $\LL_{\fp}\LL_A=0$ whenever $\fp \not \supseteq \fq$ for all primes $\fq \in A$. Thus, we propose the following example, illustrating a case where the prime $\fp$ is not visible and $\LL_{\fp} \LL_A\neq 0$.

\begin{example}\label{ex-SpcO(2)counterexamplefractureaxiom-sectionstratification}
%It is proved in \cite[\S 12]{BHS-stratification} that for any $G$ compact Lie group the rational $G$-equivariant stable homotopy category $G\da \CS p_{\LQ}$ is stratified. The proof is based on the results developed in \cite{gre-rationalbalmer}. In particular, \cite[Thm.~10.2]{gre-rationalbalmer} establishes a homeomorphism 
%\begin{align*}
%(\Gamma G, zf)&\rightarrow \text{Spc}(G\da \CS p_{\LQ}^c)\\
%[H]&\mapsto \fp_H=\{X\in G\da \CS p_{\LQ}^c : \phi^HX\cong 0\}
%\end{align*}
%where $\Gamma G$ is the set of conjugacy classes of closed subgroups of $G$ endowed with the $zf$-topology, while $\phi^H$ denotes the geometric $H$-fixed point functor. Moreover, we have an inclusion of Balmer primes $\fp_H\leq \fp_K$ if and only if we have a cotoral inclusion $H\leq_{ct}K$ (see \cite[Cor.~5.4]{gre-rationalbalmer} and \cite[Cor.~7.4]{gre-rationalbalmer}).
We saw in Example~\ref{ex-GSpQstratified-sectionbalmer} that $G\da \CS p_{\LQ}$ is a stratified tensor triangulated category. We claim that for $G=O(2)$ the conclusion of Proposition~\ref{prop-fractureaxiombalmerspectrum} fails for a particular Balmer prime $\fp$ which is not visible.
 
We first show that $\text{Spc}(O(2)\da \CS p_{\LQ}^c)\cong (\Gamma O(2), zf)$ is not noetherian. It is a well-known fact but we present a proof for the benefit of the reader.

By tom Dieck's isomorphism (\cite[Ch.~V, Lemma 2.10]{LMSMcC-equivstabhomth}) we have 
\begin{equation} \label{eq-tomdieckisoO2-sectionbousloc}
\pi_0^{O(2)}(S^0)\cong C(\Phi O(2),\LQ)
\end{equation}
and this allows us to split the category of rational $O(2)$-spectra in two parts. Here $\Phi O(2)$ denotes the set of closed subgroups of $O(2)$ with finite Weyl group endowed with the $f$-topology (\cite[Ch.~V, Not. 2.3]{LMSMcC-equivstabhomth}); it consists of an isolated point $[SO(2)]$ and a sequence $\{[D_{2n}] : n\geq 1\}$ converging to $[O(2)]$. Under the isomorphism (\ref{eq-tomdieckisoO2-sectionbousloc}) these two separated subspaces correspond to orthogonal idempotent elements $e_{\CC}, e_{\CD}\in \pi^{O(2)}_0(S^0)$ which give us a decomposition of the whole tensor triangulated category
\[ O(2)\da \CS p_{\LQ}\simeq e_{\CC}(O(2)\da \CS p_{\LQ})\times e_{\CD}(O(2)\da \CS p_{\LQ}).\]
The component $e_{\CC}(O(2)\da \CS p_{\LQ})$ is called the \textit{toral part} of the category. Restricting the $O(2)$-action along $\LT=SO(2)$ provides us with a functor
\[ e_{\CC}(O(2)\da \CS p_{\LQ})\rightarrow \LT\da \CS p_{\LQ}. \]
While forgetting the higher homotopy structure we obtain an action of the residual group $W=O(2)/SO(2)$ on the spectra of the target category. 

Usually this would lead to a loss of information: the $G$-equivariant stable homotopy category is much more than the category of non-equivariant spectra enriched with a $G$-action. But in this situation it can be proved that this forgetful functor actually induces an equivalence
\[ e_{\CC}(O(2)\da \CS p_{\LQ})\simeq \LT \da \CS p_{\LQ}[W].  \]
\cite[Cor.~4.22]{barnes-rationalO2equiv} implies that this equivalence is actually symmetric monoidal.

As far as Balmer spectra and localizations are concerned, this additional action of $W$ has no effect. Hence, the component of the Balmer spectrum provided by $e_{\CC}(O(2)\da \CS p_{\LQ})$ is homeomorphic to $\text{Spc}(\LT\da \CS p_{\LQ}^c)$, which is (non-canonically) isomorphic to $\text{Spc}(D(\mathbb{Z})^c)$ (see Example~\ref{ex-counterexformulaloclim-sectionfinite} below).

Instead, for the \textit{dihedral part} $e_{\CD}(O(2)\da \CS p_{\LQ})$ we have $D_{2n}\not \leq_{ct} O(2)$, hence the corresponding primes form a Balmer spectrum homeomorphic to $\big\{  \frac{1}{n} : n\geq 1 \big\} \cup \{0\}$ which we saw to be not noetherian in Example~\ref{ex-notnoetherianspectalspace-sectionbalmer}. Thus, from now on we focus on this part of the category of rational $O(2)$-spectra.

If the claim of Proposition~\ref{prop-fractureaxiombalmerspectrum} were true, then it would hold that for any subset $A\subseteq \text{Spc}(O(2)\da \CS p_{\LQ}^c)$ consisting of conjugacy classes of finite dihedral subgroups we have $\LL_{O(2)}\LL_{A}=0$. This is not the case if $A$ is infinite. Indeed, for any collection of finite dihedral subgroups we have
\[ \Gamma_AS^0 \cong \bigvee_{[H]\in A}e_{H}S^0\]
where $e_{H}$ is the idempotent in $\pi_0^{O(2)}(S^0)$ corresponding to the isolated point $[H]$ under (\ref{eq-tomdieckisoO2-sectionbousloc}). 

The spectra $e_{H}S^0$ are compact and self-dual. Indeed, by construction we have a splitting $S^0\cong e_HS^0\vee Z$ where $\phi^H(e_HS^0)\cong S^0$, $\phi^HZ=0$, while for a subgroup $K$ not conjugate to $H$ we have $\phi^K(e_HS^0)=0$, $\phi^KZ\cong S^0$. Being summands of a compact object, both $e_HS^0$ and $Z$ are compact. Taking the dual of the splitting gives us $S^0\cong De_HS^0\vee DZ$. For any equivariant compact spectrum taking the dual does not change isotropy: apply \cite[Prop.~2.7~(i)]{balmer-suppfiltr} to the description of Balmer primes given in \cite[Thm.~8.4]{gre-rationalbalmer}. Thus, smashing with $e_HS^0$ the previous equality gives us
\[ e_HS^0\cong e_HDe_HS^0 \wedge e_HDZ\cong e_HDe_HS^0\cong De_HS^0. \]
This allows to us to present an easy description of $\LL_A$.  Observe that the chosen $A$ is a Thomason subset of $\text{Spc}(e_{\CD}(O(2)\da \CS p_{\mathbb{Q}})^c)$. Then, it follows that 
\[\LL_AX=\Lambda_AX\cong \prod_{[H]\in A}F(e_{H}S^0, X)\cong \prod_{[H]\in A}De_{H}S^0\wedge X \cong \prod_{[H]\in A}e_{H}S^0\wedge X.\]
Let $\CF$ be the family generated by all finite dihedral subgroups, so that 
\[E\CF_+\rightarrow S^0\rightarrow \widetilde{E}\CF\]
is the isotropy separation sequence of the dihedral part of the spectra. Identifying $\CF$ with the subset of $\Gamma O(2)$ given by the classes of all finite dihedral subgroups, we have
\[ E\CF_+\cong \Gamma_{\CF}S^0\cong \bigvee_{[K]\in \CF}e_KS^0.\]
Smashing this object with $\Lambda_AX$ and using the fact that the spectra $e_{K}S^0$ are compact self-dual we obtain
\begin{align*}
E\CF_+\wedge \Lambda_AX=&\bigvee_{[K]\in \CF}e_KS^0\wedge \prod_{[H]\in A}e_{H}X  \cong \bigvee_{[K]\in \CF}F\bigg(e_KS^0, \prod_{[H]\in A}e_{H}X\bigg)\cong \\ &\bigvee_{[K]\in \CF} \prod_{[H]\in A}F(e_KS^0, e_{H}X)  \cong \bigvee_{[K]\in \CF} \prod_{[H]\in A} e_KS^0\wedge e_HX\cong \\ &\bigvee_{[H]\in A} e_HX.
\end{align*}
Thus, we deduce that smashing $\Lambda_AX$ with the above isotropy separation sequence we obtain the following exact triangle
\[\bigvee_{[H]\in A}e_{H}X\rightarrow \prod_{[H]\in A}e_{H} X\rightarrow \widetilde{E}\CF\wedge\Lambda_AX.\]
But we have $\LL_{O(2)}\cong \widetilde{E}\CF\wedge -$, therefore  $\LL_{O(2)}\LL_AX$ is the obstruction to the infinite wedge and the infinite product of the objects of $e_{H}X$ coinciding. This is generally not trivial for $A$ infinite.

Indeed, we may compare the homotopy groups
\[ \bigg[S^0, \bigvee_{[H]\in A}e_HX \bigg] \cong \bigoplus_{[H]\in A}[S^0, e_HX]\rightarrow \prod_{[H] \in A}[S^0, e_HX]\cong \bigg[ S^0, \prod_{[H]\in A}e_HX \bigg]. \]
Consider the case $X=S^0$. Since $[S^0, e_HS^0]$ corresponds in $C(\Phi O(2), \mathbb{Q})$ under tom Dieck isomorphism to the span generated by the characteristic function of $[H]$, it must be isomorphic to $\mathbb{Q}$. Thus, the two terms above differ if $A$ is infinite.
\end{example}

Another valid question is if Proposition~\ref{prop-fractureaxiombalmerspectrum} could be improved. That is, it could happen that $\LL_{\fp}\LL_{\fq}=0$ even if $\fp \supset \fq$. We present the next example to show that in a tensor triangulated category with noetherian Balmer spectrum we should expect further refinements to be false in general.
\begin{example}\label{ex-BPcomputationthread-sectionbasics}
We recall Example~\ref{ex-LnSpstratified-sectionbalmer} and consider the tensor triangulated category $L_n \CS p_{(p)}$. The objects $g(\CE_i)=M_{i}S^0$ and $K(i)$ have the same Bousfield class in $L_n \CS p_{(p)}$. Hence, $\LL_i$ is the classical localization with respect to the $i$-th Morava K-theory.

We denote by $BP$ the $p$-local Brown-Peterson spectrum, so that $\pi_*(BP)\cong \mathbb{Z}_{(p)}[v_1,v_2,\dots]$ where $v_i$ are the Araki generators with $|v_i|=2(p^i-1)$. As standard, for a natural number $m$ we denote by $I_m$ the ideal generated by $v_0, v_1,\dots, v_{m-1}$, with $v_0=p$. Let $BP_*\da \text{Mod}$ be the category of graded $BP_*$-modules. For any $m$ we can define the following functor
\[\phi_m\colon BP_*\da \text{Mod}\rightarrow BP_*\da \text{Mod} \qquad M\mapsto (v_m^{-1}M)^{\wedge}_{I_m}, \]
i.e. we first invert the element $v_m$ and then complete with respect to the ideal $I_m$. We will write an arbitrary subset $A\subseteq \{0,1,\dots, n\}$ in the form $A=\{a_1<a_2<\dots<a_k\}$ and we define correspondingly $\phi_A=\phi_{a_1}\phi_{a_2}\dots \phi_{a_k}$. 

Let us consider the truncated spectrum $BP\bc{n}=BP/(v_{n+1}, v_{n+2},\dots)$ which has homotopy groups $\pi_*(BP\bc{n})=\mathbb{Z}_{(p)}[v_1,v_2,\dots, v_n]$. It was proved in \cite[Prop.~4.4.9]{thesis} that we have an isomorphism of $BP_*$-modules
\begin{equation} \label{eq-piLBP} \pi_*(\LL_{a_1}\LL_{a_2}\dots\LL_{a_k}BP\bc{n})\cong \phi_A(BP\bc{n}_*).   \end{equation}
Since this graded module is not trivial for any subset $A$, any further refinement of Proposition~\ref{prop-fractureaxiombalmerspectrum} is not possible.

The proof of \cite[Prop.~4.4.9]{thesis} relies on the computations of \cite[Lemma~2.3]{hovey-bousfield-loc-splitting-conj} and the surrounding discussion. \cite[Lemma~2.3]{hovey-bousfield-loc-splitting-conj} specifically proves $\pi_*(\LL_{m}BP)\cong \phi_m(BP_*)$ but the proof can be adapted to any $BP$-algebra $E$ such that the sequence $v_0, v_1, \dots, v_m$ is regular in $\pi_*(E)$ and $\pi_*(E)/I_m\neq 0$. This way $v_m^{-1}E$ results a $v_m$-periodic Landweber exact spectrum and starting from this the claim can be proven by induction on $|A|$. 

We consider $BP\bc{n}$ so that we can restrict to $BP\bc{n}_*$-modules, where $BP\bc{n}_*$ is noetherian. The noetherianity of the base graded ring $R$ guarantees by the Artin-Rees lemma that if a sequence $x_1,\dots, x_l$ is regular in $R$ and $I\subset R$ is a homogeneous ideal then the same sequence remains regular in the completed ring $R^{\wedge}_I$.
\end{example}

From now on, motivated by Proposition~\ref{prop-fractureaxiombalmerspectrum}, we will consider only stratified tensor triangulated categories with noetherian Balmer spectrum.

Proposition~\ref{prop-fractureaxiombalmerspectrum} states that a composition of localizations $\LL_{\mathbb{A}}$ contains non-trivial information only if there are inclusions between the primes belonging to successive indexing subsets. This intuition motivates the introduction of the following Conjecture~\ref{conj} which states that these chains of prime inclusions identify two compositions of localizations.
\begin{definition}
Let $P$ be a general poset. A \textit{chain} in $P$ consists in a subset $C\subseteq P$ such that the induced order on $C$ is total. We denote by $s(P)$ the set of chains in $P$, including the empty chain. We endow this set with the inclusion ordering, i.e. $C_1\leq  C_2$ if and only if $C_1\subseteq C_2$. Furthermore, we set $s(P)'$ to be the set of non-empty chains.

Given a chain $C$ we define its \textit{dimension} to be $\text{dim}(C)=|C|-1$. We define the \textit{(Krull) dimension} of the poset $P$ to be the supremum of $\text{dim}(C)$ for $C$ ranging among all the chains of $P$.

Given a point $p\in P$ we set its \textit{length} to be the supremum among the dimensions of those chains whose maximal object is $p$. That is, $p$ has finite length $n\in \mathbb{N}$ if there exists a chain $p_0<p_1<\dots<p_{n-1}<p$ and there is not a chain  $p_0<p_1<\dots<p_{n-1}<p_n<p$. If there are chains of arbitrary dimension terminating with $p$, then this point has infinite length.

Let $\mathbb{A}=(A_1,\dots, A_k)$ a $k$-tuple of subsets of $P$, we say that the $k$-tuple $(a_1,\dots, a_k)$ is a \textit{thread} for $\LA$ if $a_i\in A_i$ for  $1\leq i \leq k$ and $a_1\geq a_2\geq \dots \geq a_k$ holds. A chain $T\subseteq P$ is called a \textit{thread set} for $\LA$ if there exist elements $a_i\in T\cap A_i$ forming a thread $(a_1,\dots, a_k)$ for $\LA$. We denote by $T(\LA)$ the collection of thread sets for $\LA$.
\end{definition}

\begin{conjecture}\label{conj}
Let $\CT$ be a rigidly-compactly generated tensor triangulated category which is stratified via the Balmer-Favi support and such that its Balmer spectrum $\emph{Spc}(\CT^c)$ is noetherian and has finite Krull dimension. Suppose we have $\mathbb{A}=(A_1,\dots, A_k)$ and $\mathbb{B}=(B_1,\dots, B_l)$ two tuples of subsets of $\emph{Spc}(\CT^c)$. 

If $T(\LA)=T(\LB)$ then there exists a canonical natural isomorphism $\LL_{\LA}\cong \LL_{\LB}$.
\end{conjecture}
To conclude the section, we make the following observations on the Conjecture.
\begin{remark}
The requirement $T(\LA)=T(\LB)$ spelled out means that for every $T\subseteq \text{Spc}(\CT^c)$ we have the tuple $(T\cap A_1,\dots, T\cap A_k)$ admits a thread $\fp_1\supseteq \dots \supseteq \fp_k$ if and only if $(T\cap B_1,\dots, T\cap B_l)$ admits a thread $\fq_1\supseteq \dots \supseteq \fq_l$. However, the two decreasing sequences do not have to be the same and the primes $\fp_i$ and $\fq_j$ do not have to be related in any specific way. Only their existence is required.
\end{remark}
\begin{remark}\label{rmk-isoconjcanonical-sectionbasics}
The isomorphism $\LL_{\LA}\cong \LL_{\LB}$ should be canonical in the sense that it should be compatible with all the natural transformations between iterated localizations arising from their intrinsic properties. Thus, it will make commute all the reasonable diagrams presenting such transformations.

For example, if we have a tuple $\LC=(C_1,\dots, C_m)$ such that there are indexes $1\leq a_1\leq a_2\leq \dots \leq a_m\leq k$ with $A_{a_i}\subseteq C_i$, then there exists an induced natural transformation $\LL_{\LC}\Rightarrow \LL_{\LA}$.  Similarly, if there are indexes $1\leq b_1\leq \dots \leq b_m\leq l$ such that $B_{b_j}\subseteq C_j$, then we have $\LL_{\LC}\Rightarrow \LL_{\LB}$. In this case we require the above isomorphism to fit in the commutative diagram
\begin{center}
\begin{tikzcd}
& \LL_{\LC}\arrow[dl, Rightarrow] \arrow[dr, Rightarrow] & \\
 \LL_{\LA} \arrow[rr, Rightarrow, "\cong" above]& &\LL_{\LB}.
\end{tikzcd}
\end{center}
In practice, the invoked isomorphism will arise from the universal properties of the localizations or expressing the compositions of localizations as the homotopy limit of a diagram of other iterated localizations. Therefore, the commutativity of all diagrams we require will be automatic.
\end{remark}

\begin{remark}
We do not claim that the condition $T(\LA)=T(\LB)$ should be necessary to have $\LL_{\LA}\cong \LL_{\LB}$. It is not clear in full generality why two isomorphic compositions should have the same thread sets, even if this seems to be the case in all concrete examples where computations are feasible. 

In Example~\ref{ex-BPcomputationthread-sectionbasics} the formula (\ref{eq-piLBP})
%\[ \pi_*(\LL_{a_1}\dots\LL_{a_k}BP\bc{n})\cong \phi_A(BP\bc{n}_*)   \]
implies the iterated localizations $\LL_{a_1}\dots\LL_{a_k}$ are distinct for different increasing sequences of indexes since the graded modules $\phi_A(BP\bc{n}_*)$ can be proven to be all distinct.

Also, see Remark~\ref{rmk-dim1samethreadssufficient-sectionlowdim} below for observations in the case where the Balmer spectrum has dimension 1.
\end{remark}

\section{The first reductions}\label{section-firstreductions}
We now provide starting results which allow us to reduce the subsets of a tuple $\LA$ without changing the isomorphism class of the associated iterated localization $\LL_{\LA}$. 

\begin{remark}\label{rmk-categorywithmodel-sectionreduction}
In the following proofs we will need to form homotopy limits: to do this from now on we will assume our tensor triangulated category $\CT$ admits a model. That is, it is the homotopy category $\text{Ho}(\CC)$ of $\mathcal{C}$ a stable monoidal quasi-category or a stable monoidal model category. This will guarantee a well-behaved calculus of homotopy limits. I.e., we can construct a strong stable derivator with underlying category $\CT$ so that the homotopy limits follow the properties of Kan extensions in derivator theory (functoriality, Kan formula, et cetera). See \cite[Def.~1.10]{groth-pointed-der}, \cite[Example~1.5]{groth-pointed-der} and \cite[Example~1.6]{groth-pointed-der}.

Since in mathematical practice the most important stratified tensor triangulated categories admit a model, this assumption is not restrictive. Indeed, the categories presented in Example~\ref{ex-DRnoethstratified-sectionbalmer} (\cite[Thm.~2.3.11]{hovey-model}), Example~\ref{ex-LnSpstratified-sectionbalmer} (\cite[Thm.~4.1]{ekmm} and \cite[Cor.~4.10]{ekmm}) and Example~\ref{ex-GSpQstratified-sectionbalmer} (\cite[Thm.~4.2]{equivorthspe}) are the homotopy categories of a model. 

Moreover, we will liberally use the fact that given a tensor triangulated category admitting a model, then its localization still admits a model. For references see \cite[Thm.~4.1.1]{hirch} and \cite[Prop.~5.5.4.15]{highertopos}.
\end{remark}

\begin{lemma}\label{lem-fracturecube-sectionreduction}
Let $\CT$ be a stratified tensor triangulated category with noetherian Balmer spectrum and let $A_i\subseteq\emph{Spc}(\CT^c)$ for $1\leq i \leq n$ be subsets of the Balmer spectrum. Suppose the following property holds: fixed $i$ and fixed a collection $J$ of indexes $j<i$ we have that  $\forall X \in \CT$  the equality $\LL_{A_{i}}X=0$ implies $\LL_{A_{i}}\LL_{B}X=0$ for $B=\bigcup_{j\in J}A_j$.

We set $A=\bigcup_{i=1}^nA_i$. Then, $\LL_{A}X$ can be expressed as the homotopy limit of the punctured cube
\begin{align*}
\CP(\{1,2,\dots, n\})'&\rightarrow \CT\\
S=\{i_1<\dots<i_k\}&\mapsto \LL_{A_{i_1}}\LL_{A_{i_2}}\dots \LL_{A_{i_k}}X.
\end{align*}
\end{lemma}
\begin{proof}
This is the usual argument for the chromatic fracture cube. It states that given the Bousfield classes $\bc{K(1)},\dots, \bc{K(n)}$ such that $\LL_{K(i)}X=0$ implies the equality $\LL_{K(i)}\LL_{E(J)}X=0$, where $J$ is an arbitrary set of indexes $j<i$ and  $\bc{E(J)}=\bigvee_{j\in J}\bc{K(j)}$,  then the localization of $X$ with respect to $\bigvee_{i=1}^n\bc{K(i)}$ can be expressed as the homotopy limit of the cube presenting at its vertices the compositions of localizations $\LL_{K(i_1)}\dots \LL_{K(i_k)}X$ for all the increasing sequences $i_1<\dots<i_k$. A complete proof is presented in \cite[Prop.~7.12]{itloc}.

In our case $K(i)=g(A_i)$ and $\coprod_{i=1}^nK(i)=g(A)$.
\end{proof}
\begin{remark}\label{rmk-conditionfracturecube-sectionreduction}
The crucial condition that $\LL_{A_i}X=0$ implies $\LL_{A_i}\LL_{B}X=0$ for $B=\bigcup_{j\in J}A_j$  where $J$ is a set of indexes $j<i$ is easily seen to be satisfied in these two situations:
\begin{itemize}
\item[(i)] when we take the subsets $A_i$ in such a way that fixed an index $i$ for any $j<i$ it holds that $\forall \fp \in A_i, \fq \in A_j$ we have $\fp\not \supseteq \fq$. Then, Corollary~\ref{cor-facturenoetherianbalmerspectrum-sectionstratification} implies the condition is satisfied.

\item[(ii)] When all the $A_i$'s are complements of Thomason subsets. In this case the localizations $\LL_{A_i}=L_{A_i}$ are smashing. Also, $B=\bigcup_{j\in J}A_j$ is the complement of a Thomason subset for any set of indexes $J$, hence $\LL_B$ is smashing as well. We conclude $\LL_{A_i}$ and $\LL_B$ commute.
\end{itemize}
\end{remark}

\begin{proposition}\label{prop-LALB=LAcapgeqBLB-sectionreduction}
Let $\CT$ be a stratified tensor triangulated category with noetherian Balmer spectrum. Let $A, B$ be two subsets of $\emph{Spc}(\CT^c)$, then we have 
\[ \LL_A\LL_B\cong \LL_{A\cap[\supseteq B]}\LL_B \]
where we recall $[\supseteq B]=\{ \fp \in  \emph{Spc}(\CT^c) : \exists \fq \in B \ \fp \supseteq \fq \}$.
\end{proposition}
\begin{proof}
Set $A'=A\cap[\supseteq  B]$ and $A''=A\setminus A'$. It follows that $\forall \fp \in A''$ there is not a prime $\fq \in A'$ such that $\fp \supseteq \fq$. Otherwise, by definition, there would exist $\mathfrak{r} \in B$ with $\fq \supseteq \mathfrak{r}$ and it would follow $\fp \supseteq \fq \supseteq \mathfrak{r}$, hence $\fp \in A'$.

Corollary~\ref{cor-facturenoetherianbalmerspectrum-sectionstratification} states $\LL_{A''}	\LL_{A'}=0$, hence by Lemma~\ref{lem-fracturecube-sectionreduction} there exists for any object $X$ a homotopy pullback square of the form
\begin{center}
\begin{tikzcd}
\LL_{A}X \arrow[r] \arrow[d] & \LL_{A'}X \arrow[d] \\
\LL_{A''}X \arrow[r] & \LL_{A'}\LL_{A''}X. \arrow[ul, phantom, "\ulcorner" very near end]
\end{tikzcd}
\end{center}
If we consider this square for $X=\LL_BX'$, then Corollary~\ref{cor-facturenoetherianbalmerspectrum-sectionstratification} yields that the bottom row becomes an isomorphism between two zero objects. Therefore, the upper horizontal natural transformation $\LL_A\LL_B \Rightarrow \LL_{A'}\LL_B$ is an isomorphism.
\end{proof}

\begin{proposition}\label{prop-LALB=LALBcupnotleqA-sectionreduction}
Let $\CT$ be a stratified tensor triangulated category with noetherian Balmer spectrum. Let $A,B$ be any two subsets of $\emph{Spc}(\CT^c)$, then it holds
\[ \LL_A\LL_B\cong \LL_A\LL_{B \cup [\not \subseteq A]} \]
where $[\not \subseteq A]=\{ \fp \in \emph{Spc}(\CT^c) : \forall \fq \in A \ \fp \not \subseteq \fq  \}$.
\end{proposition}
\begin{proof}
We set $C=[\not\subseteq A]$ and $B'=B\cup C$.

We claim that for any object $X \in \CT$ there exists a homotopy pullback square
\begin{center}
\begin{tikzcd}
\LL_{B'}X\arrow[d] \arrow[r] & \LL_{C}X \arrow[d]\\
\LL_{B}X \arrow[r] & \LL_{C}\LL_{B}X. \arrow[ul, phantom, "\ulcorner" very near end]
\end{tikzcd}
\end{center}
We use Lemma~\ref{lem-fracturecube-sectionreduction}. Suppose $\LL_{B}X=0$ and let us show this implies $\LL_{B}\LL_{C}X=0$. By construction $C$ is the complement of a Thomason subset, hence the localization $\LL_{C}=L_C$ is smashing. Therefore, we have to prove  $\LL_{B}\LL_{C}X=\LL_{B}(X\otimes \LL_{C}\mbu)$ is trivial. But by assumption $X$ is $g(B)$-acyclic, thus $X\otimes \LL_{C}\mbu$ is $g(B)$-acyclic as well.
%Indeed, since all objects in the square are $E\bc{B'}$-local, it is enough to show that the homotopy fibre is $E\bc{B'}$-acyclic. This is equivalent to proving that after applying $\LL_{C}$ and $\LL_B$ separately the square becomes homotopy pullback in both cases.
%
%If we apply $\LL_C$ the lower horizontal map becomes the equivalence $\LL_C\LL_BX\cong \LL_C\LL_BX$, while $C\subseteq B'$ implies $\LL_C\LL_{B'}\cong \LL_C$ thus also the upper horizontal morphism becomes an equivalence.
%
%Now we apply $\LL_B$. As before, $B\subseteq B'$ implies that the left vertical map becomes an equivalence $\LL_B\LL_{B'}\cong \LL_B$. To conclude we examine the right vertical map. Since $C$ is the complement of a family, by noetherianity of the Balmer spectrum it is the complement of a Thomason subset thus $\LL_C=L_C$ is a smashing localization. Hence the fibre of $\LL_CX\rightarrow \LL_C\LL_BX$ coincides with  $L_{C}\mathbbm{1}\otimes \mathbb{V}_BX=0$ (here $\mathbb{V}_B$ is the acyclization with respect to $g\bc{B}$) and this becomes zero once tensored with $g(B)$, therefore $\LL_B\LL_C\mathbb{V}_BX=0$.

Since we established the above square is homotopy pullback, we can apply now $\LL_A$ and the diagram still remains homotopy cartesian. But by Corollary~\ref{cor-facturenoetherianbalmerspectrum-sectionstratification} we have $\LL_A\LL_C=0$, therefore the left vertical morphism provides us with the natural equivalence $\LL_A\LL_{B'}\cong \LL_A\LL_B$ as we wanted.
\end{proof}
\begin{corollary}\label{cor-LambdaALambdaB=LambdaALambdaBcapleqA-rationaltorus}
Let $\CT$ be a stratified tensor triangulated category with noetherian Balmer spectrum. Let $A,B$ be two subsets of $\emph{Spc}(\CT^c)$, then we have
\[ \LL_A\LL_B\cong \LL_A\LL_{B\cap [\subseteq A]}. \]
\end{corollary}
\begin{proof}
Observe that $B\setminus [ \not \subseteq A ]= B\cap [\subseteq A]$ and $(B \setminus [ \not \subseteq A ]) \cup [ \not \subseteq A ]=B \cup [ \not \subseteq A ]$. Then, apply Proposition~\ref{prop-LALB=LALBcupnotleqA-sectionreduction} to obtain
\[ \LL_{A}\LL_B\cong \LL_{A}\LL_{B\cup [ \not \subseteq A ]}=\LL_{A}\LL_{(B \setminus [ \not \subseteq A ]) \cup [ \not \subseteq A ]}\cong \LL_A\LL_{B\cap [\subseteq A]}. \]
\end{proof}

\begin{definition}
Let $(P,\leq)$ be a poset. Let $\mathcal{N}(P)$ be the set of all $k$-tuples of subsets of $P$ for all $k\geq 1$. More precisely,
\[ \mathcal{N}(P)=\bigcup_{k\geq 1}\mathcal{P}(P)^k.\]

We say that an arbitrary tuple $\LA=(A_1,\dots, A_k)$ is
\begin{itemize}
\item[(i)] \textit{upward concatenated} if $A_{i+1}\subseteq [\leq A_{i}]$ for all $1\leq i \leq k-1$;
\item[(ii)] \textit{downward concatenated} if $A_i\subseteq [\geq A_{i+1}]$ for all $1\leq i \leq k-1$;
\item[(iii)] \textit{concatenated} if it is both upwards and downward concatenated;
\item[(iv)] \textit{collapsed} if it holds $A_i \not \subseteq A_{i+1}$ and $A_{i}\not \supseteq A_{i+1}$ for any $1\leq i\leq k-1$.
\end{itemize}
\end{definition}

\begin{definition}\label{def-taubeta-sectionreduction}
We define the following two functions on $\CP(P)^k$, for $(P,\leq)$ a poset.
\begin{align*}
\tau \colon \CP(P)^k&\rightarrow \CP(P)^k \\
(A_1,\dots, A_k)&\mapsto (A_1',A_2',\dots, A_k')
\end{align*}
where we set recursively $A_1'=A_1$ and $A_{i+1}'=A_{i+1}\cap[\leq A_{i}']$ for $1\leq i \leq k-1$.
\begin{align*}
\beta \colon \CP(P)^k&\rightarrow \CP(P)^k \\
(B_1,\dots, B_k)&\mapsto (B_1',B_2',\dots, B_k')
\end{align*}
where we set recursively $B_k'=B_k$ and $B_{k-i}'=B_{k-i}\cap[\geq B_{k-i+1}']$ for $1\leq i \leq k-1$.

Putting together all these functions for the various $k\geq 1$ we obtain two functions $\tau$ and $\beta$ on $\CN(P)$.
\end{definition}
\begin{lemma}\label{lem-betataucommute-sectionreduction}
The function $\tau$ defines a retraction of $\CN(P)$ on its subset of upward concatenated tuples. The function $\beta$ defines a retraction of $\CN(P)$ on its subset of downward concatenated tuples.

Moreover, $\beta$ and $\tau$ commute.
\end{lemma}
\begin{proof}
The first two claims are trivial. It remains to show that $\beta$ and $\tau$ commute.

Observe that if $\mathbb{C}=\tau(\mathbb{A})$, then we have 
\[ C_i=\{ p_i \in A_i : \exists p_j \in A_j \ 1\leq j \leq i-1 \quad  p_1\geq p_2\geq \dots \geq p_{i-1}\geq p_i \}. \]
Similarly, $\mathbb{B}=\beta(\mathbb{A})$ can be alternatively defined by the sets
\[ B_i=\{ p_i \in A_i : \exists p_j \in A_j \ i\leq j \leq k \quad p_i\geq p_{i+1}\geq \dots \geq p_k  \}. \]
Using these descriptions, it is immediate that both $\tau \beta(\mathbb{A})$ and $\beta \tau (\mathbb{A})$ coincide with the tuple $\mathbb{D}$ where 
\[ D_i=\{ p_i \in A_i : \exists p_j \in A_j \ 1 \leq j \leq k, j \neq i \quad p_1 \geq p_2 \geq  \dots \geq p_i\geq \dots \geq p_k  \}. \]
\end{proof}
\begin{definition}
We define the function $\delta\colon \CN(P)\rightarrow \CN(P)$ by $\delta=\beta\tau=\tau\beta$. By Lemma~\ref{lem-betataucommute-sectionreduction} this is clearly a retraction of $\CN(P)$ on its subset of concatenated tuples.

Furthermore, it admits an explicit description as $\delta(A_1,\dots, A_k)=(A_1',\dots, A_k')$ with
\[A_i'=\{a_i\in A_i : \exists a_j \in A_j \ 1 \leq j \leq k, j \neq i \quad a_1 \geq a_2 \geq  \dots \geq a_i\geq \dots \geq a_k \}\]
for any index $ 1\leq i \leq k$.
\end{definition}

\begin{lemma}\label{lem-defgamma-sectionreduction}
There exists a unique function $\gamma\colon \CN(P)\rightarrow \CN(P)$ which is a retraction on the subset of collapses tuples and such that for any $k$-tuple $(A_1,\dots, A_k)$ we have
\[ \gamma(A_1,\dots, A_k)=\gamma(A_1,\dots, A_{i-2},A_i, A_{i+1},\dots, A_k) \quad \text{if} \ A_i\subseteq A_{i-1} \]
and
\[ \gamma(A_1,\dots, A_k)=\gamma(A_1,\dots, A_{i-1},A_i, A_{i+2},\dots, A_k) \quad \text{if} \ A_i\subseteq A_{i+1}. \]
\end{lemma}
\begin{proof}
It is easy to see that in finitely many steps we can reduce a $k$-tuple to a collapsed $l$-tuple with $l\leq k$  using the operations described above of removing a subset if it is adjacent to another subset which is contained in the former.

If $A_i$ is contained in both $A_{i+1}$ and $A_{i-1}$, then it is trivial to observe that the two operations of removing $A_{i+1}$ and $A_{i-1}$ commute. Hence, the order is irrelevant and the resulting $\gamma$ is well-defined.
\end{proof}

\begin{lemma}
The function $\gamma$ preserves upward and downward concatenated tuples. Therefore, we have the equality $\gamma\delta=\delta\gamma\delta$ and this  defines a retraction on $\CN(P)$ of its subset of collapsed concatenated tuples.
\end{lemma}
\begin{proof}
We show that $\gamma$ preserves upward concatenated tuples, the proof that it preserves downward concatenated tuples is similar.

Using the description of $\tau$ in Lemma~\ref{lem-betataucommute-sectionreduction}, it is easy to see that $\LA=(A_1,\dots, A_k)$ being upward concatenated means that $\forall 1\leq i \leq k$ and $\forall a\in A_i$ there exists a sequence $a_1\geq a_2\geq \dots \geq a_{i-1}\geq a$ where $a_j\in A_j$ for any $j$ between $1$ and $i-1$.

Suppose for the tuple $\LA$ we have an inclusion $A_l\subseteq A_{l+1}$ for some $l$. If we remove $A_{l+1}$ from $\LA$ we can see the new tuple maintains the above property: if $i\leq l$ there is nothing to prove since we can use the same descending chain as before for any $a\in A_i$. If $i>l+1$ then take $a\in A_i$. Because $\LA$ is upward concatenated we have a sequence $a_1\geq \dots
\geq a_l \geq a_{l+1}\geq \dots \geq a$, by removing the element $a_{l+1}$ we obtain a new thread of the tuple $(A_1,\dots, A_l, A_{l+2},\dots, A_i)$ terminating in $a$.

Similarly, if we have an inclusion $A_l\supseteq A_{l+1}$ we can see that eliminating $A_l$ from the tuple preserves the property of having each element being the terminal point of a descending chain of elements of the previous subsets as above.
%
% This is because a thread $a_1\geq \dots \geq a_{i}\geq a_{i+1}\geq a_{i+2}$ with $a_j\in A_j$ can be reduced to a thread $a_1\geq \dots \geq a_{i}\geq a_{i+2}$ and vice versa such a thread gives rise to the thread $a_1\geq \dots \geq a_{i}\geq a_{i}\geq a_{i+2}$ where the second copy of $a_i$ belongs to $A_{i+1}$ by the condition $A_i\subseteq A_{i+1}$.
%
%We can similarly check that collapsing the inclusion $A_{i}\supseteq A_{i+1}$ before and after applying $\tau$ does not change the final result. It follows $\gamma\tau=\tau\gamma$ as claimed.
\end{proof}

\begin{remark}
It is easy to understand that $\delta$ does not preserve collapsed tuples: consider three distinct elements $p,q,r$ of a poset such that any two of them are not related by the partial order, then the pair $(\{p, r\}, \{q, r\})$ is collapsed, but $\delta(\{p, r\}, \{q, r\})=(\{r\},\{r\})$.
\end{remark}

\begin{proposition}\label{prop-LA=LdeltaA-sectionreduction}
Let $\LA$ be a tuple of subsets of $\emph{Spc}(\CT^c)$. Then, there is a natural isomorphism $\LL_{\LA}\cong \LL_{\gamma\delta(\LA)}$.
\end{proposition}
\begin{proof}
It is a matter of understanding that the definitions of the operators $\tau$ and $\beta$ correspond to the reductions of respectively Corollary~\ref{cor-LambdaALambdaB=LambdaALambdaBcapleqA-rationaltorus} and Proposition~\ref{prop-LALB=LAcapgeqBLB-sectionreduction}. Thus, applying $\delta$ does not change the isomorphism type of the iterated localization.

Meanwhile, the reduction $\gamma$ corresponds to the isomorphism $\LL_{A}\LL_{B}\cong \LL_{A}\cong \LL_{B}\LL_{A}$ whenever we have an inclusion $A\subseteq B$.
\end{proof}

This result allows us to prove Conjecture~\ref{conj} in a specific case.
\begin{corollary}
Let $\LA$ be a tuple of subsets admitting no thread, then the associated composition of localizations $\LL_{\LA}$ is trivial. 
\end{corollary}
\begin{proof}
Since $T(\LA)=\emptyset$, it follows that $\delta(\LA)=\emptyset$. Now apply Proposition~\ref{prop-LA=LdeltaA-sectionreduction}.
\end{proof}

\begin{remark}
Proposition~\ref{prop-LA=LdeltaA-sectionreduction} is not strong enough to prove Conjecture~\ref{conj}. There are many cases in which two different collapsed concatenated tuples have the same thread sets.

For example, suppose we have distinct Balmer primes $\fp_1, \fp_2, \fq_1, \fq_2$ such that the only inclusions among them are $\fp_1\supseteq \fp_2$ and $\fq_1\supseteq \fq_2$. In this situation the triple $(\{\fp_1, \fq_1\}, \{\fp_1, \fq_2\}, \{\fp_2, \fq_2\})$ and the pair $(\{\fp_1, \fq_1\}, \{\fp_2, \fq_2\})$ are both collapsed concatenated and their thread sets consists of all the chains containing  $\fp_1, \fp_2$ or $\fq_1, \fq_2$.
\end{remark}

\section{The finite case}\label{section-finitecase}
We now prove the isomorphism in the claim of Conjecture~\ref{conj} in the case where we are considering only finitely many elements of $\text{Spc}(\CT^c)$. This is essentially \cite[Thm.~1.8]{itloc} and \cite[Prop.~1.10]{itloc}. However, in that source the set indexing the localizations is linearly ordered since its main application was the chromatic case $L_n\CS p_{(p)}$ of Example~\ref{ex-LnSpstratified-sectionbalmer}.

We now propose the modifications to make the same argument work for a Balmer spectrum which is not necessarily linearly ordered.

We first notice that Corollary~\ref{cor-facturenoetherianbalmerspectrum-sectionstratification} provides an adaptation of the fracture axiom \cite[Def.~2.1]{itloc}, allowing us to construct homotopy limit cubes as illustrated in Lemma~\ref{lem-fracturecube-sectionreduction}.

\begin{definition}\label{def-isotropicloccomp}
Let $C \in s(\text{Spc}(\CT^c))$ be an arbitrary chain, hence we can write $C=\{ \fp_1\supset \fp_2\supset \dots \supset \fp_k \}$ for some $k \in \mathbb{N}$. We set
\[ \varphi_C=\LL_{\fp_1}\LL_{\fp_2}\dots \LL_{\fp_k}. \]
In the case $C=\emptyset$ we simply take the convention $\varphi_{\emptyset}=Id$.

Observe that by Corollary~\ref{cor-facturenoetherianbalmerspectrum-sectionstratification} this is the only composition of localizations with respect to a single Balmer prime which is possibly non-trivial.
\end{definition}

We start by proving the following analogue of the fracture cube. 
\begin{proposition}\label{prop-LambdaAholimchains-sectionfinitecase}
Let $A \subseteq \emph{Spc}(\CT^c)$ be a finite subset, then for any object $X$ of the tensor triangulated category $\CT$ we have an isomorphism 
\[ \LL_AX \cong \holim_{C \in s(A)'}\varphi_C X. \]
In particular, this equivalence holds for any subset $A$ if the Balmer spectrum is finite.
\end{proposition}
\begin{proof}
This is a slight adaptation of \cite[Prop.~7.12]{itloc}.

By the assumption on $A$ we have finitely many Balmer primes. Thus, we can label all of them $\fp_1, \fp_2,\dots, \fp_n$ with an ordering such that if $\fp_i\supset\fp_j$ then $i<j$, equivalently if $i > j$ then $\fp_i \not \supseteq \fp_j$. The above Corollary~\ref{cor-facturenoetherianbalmerspectrum-sectionstratification} states the Bousfield classes $\bc{g(\fp_i)}$ satisfy the key assumption of Lemma~\ref{lem-fracturecube-sectionreduction}. Hence, we have an isomorphism
\[ \LL_AX \cong \holim_{B \in \CP(A)'}\LL_{\fp_{b_1}}\LL_{\fp_{b_2}}\dots \LL_{\fp_{b_k}}X \]
where $B=\{\fp_{b_1}, \fp_{b_2}, \dots, \fp_{b_k} \}$ with $b_1<b_2<\dots<b_k$.

The assumption that $\CT$ admits a geometric model $\CC$ implies that this diagram can be realized in $\CC$ rather than at the homotopy level of $\CT$. More precisely, there exists an element $\CD$ of $\text{Ho}(\CC^{\CP(A)'})$ that when passed through the forgetful functor $\text{Ho}(\CC^{\CP(A)'})\rightarrow \text{Ho}(\CC)^{\CP(A)'}$ produces the above diagram over $\CP(A)'$ of iterated localizations $\LL_{\fp_{b_1}}\LL_{\fp_{b_2}}\dots \LL_{\fp_{b_k}}X$.

It is immediate to see that many of compositions of localizations of the invoked puncture cube are actually trivial: the composition $\LL_{\fp_{b_1}}\LL_{\fp_{b_2}}\dots \LL_{\fp_{b_k}}$ is zero unless $\fp_{b_1}\supset \fp_{b_2}\supset \dots \supset \fp_{b_k}$.

Observe that the inclusion $j \colon s(A)'\hookrightarrow \CP(A)'$ is the inclusion of a downward closed subset (a sieve in the sense of \cite[Def.~1.28]{groth-pointed-der}): clearly any subset of a chain of elements of $A$ must remain totally ordered.

As we explained, if $B$ is not a chain we have $\LL_{\fp_{b_1}}\LL_{\fp_{b_2}}\dots \LL_{\fp_{b_k}}X=0$. Therefore, $\CD$ belongs to the essential image of the right Kan extension $j_*\colon \text{Ho}(\CC^{s(A)'})\rightarrow \text{Ho}(\CC^{\CP(A)'})$ (see \cite[Prop.~1.29]{groth-pointed-der}). This is equivalent to the unit $\CD\rightarrow j_*j^*\CD$ being an isomorphism (\cite[Lemma~1.27]{groth-pointed-der}).

We set $\star$ to be the category with just one object and as morphisms only the identity. For any indexing category $I$ we define $\pi_I \colon I\rightarrow \star$ to be the projection. This induces a restriction functor $(\pi_I)^*\colon \text{Ho}(\CC)\cong \text{Ho}(\CC^{\star})\rightarrow \text{Ho}(\CC^I)$ which has as right adjoint the right Kan extension $(\pi_I)_*\colon \text{Ho}(\CC^I)\rightarrow \text{Ho}(\CC)$. 

The homotopy limit of a diagram indexed over $I$ coincides with  $(\pi_I)_*$, thus we have
\[ \LL_A X=(\pi_{\CP(A)'})_*\CD\cong (\pi_{\CP(A)'})_*j_*j^*\CD. \]
Right Kan extensions can be formed functorially, thus $(\pi_{\CP(A)'})_*j_*=(\pi_{s(A)'})_*$ and we obtain
\[ \LL_AX\cong (\pi_{s(A)'})_*j^*\CD=\holim_{C \in s(A)'} \varphi_C X.\]
\end{proof}
Proposition~\ref{prop-LambdaAholimchains-sectionfinitecase} provides the first step in adapting the results of \cite{itloc} to the current context. Indeed, we can prove that if we are considering only finitely many Bousfield classes and the behaviour of the compositions of the associated localizations is regulated by  an ordering on the indexing poset (Corollary~\ref{cor-facturenoetherianbalmerspectrum-sectionstratification}) all the fundamental outcomes of   \cite{itloc} hold, mutatis mutandi.

To facilitate the translation from the notation of \cite{itloc} to the one adopted here, we propose the following dictionary.
\vspace{0.4cm}
\begingroup
\def\arraystretch{1.8}
\begin{longtable}{|| p{2cm} | p {4.7cm} | p{4.7cm} ||}
\hline 
\textbf{Concept} &  \textbf{Iterated chromatic localizations}& \textbf{Localizations over a finite poset}  \\
\hline
poset indexing single Bousfield classes & $N=\{0,1,\dots, n^*-1 \}$, total order $<$ & finite $F\subseteq\text{Spc}(\CT^c)$, partial order $\supset$ \\
subsets of the poset recovering joins of  B. classes & $\mathbb{P}=\CP(N)$ & $s(F)$ \\
Bousfield localization & $\lambda_A$ for $A\subseteq N$  & $\LL_{A}$ for $A\subseteq F$  \\
fracture axiom & $\lambda_i \lambda_A=0$ if $i>\max A$ & $\LL_{\fp}\LL_A=0$ if $\forall \fq\in A$ \ $\fp\not \supseteq \fq$  \\
fracture cube & $\lambda_A X=\holim_{B \in \CP(A)'}\phi_B X$ & $\LL_AX=\holim_{B \in s(A)'}\varphi_B X$\\
relation on subsets & $A \angle B$ if $a\leq b$ \ $\forall a \in A, b \in B$& $C \angle D$ if $\fp\supseteq \fq$ \ $\forall \fp \in C, \fq \in D$\\
poset indexing compositions of localizations & $\mathbb{Q}$: upward closed subsets of $\mathbb{P}$ & $\mathbb{U}$: upward closed subsets of $s(F)$\\
map $u$ relating the posets & $u\colon \mathbb{P}\rightarrow \mathbb{Q}$ \ $uA=\{ B \in \mathbb{P} : A \subseteq B \}$ & \parbox[t]{5.5cm}{ $ u\colon s(F)\rightarrow \mathbb{U}$ \\ $uC=\{ D \in s(F): C \subseteq D \}$ } \\ 
element indexing $\lambda_A$ & {$vA$: subsets of $N$ with  non-empty intersection with $A$} &    $wA$: chains in $F$  with non-empty intersection with $A$ \\
adjunction connecting the fracture cube to the above element & \parbox[t]{4cm}{$\CP(A)' \rightleftarrows vA$ \\ left adj. $\CP(A)'\hookrightarrow vA$ \\ right adj. $vA \rightarrow \CP(A)'$ \ $B \mapsto B\cap A$} & \parbox[t]{4cm}{$s(A)' \rightleftarrows wA$ \\ left adj. $s(A)' \hookrightarrow wA$ \\ right adj. $wA \rightarrow s(A)'$ \ $B\mapsto B \cap A$ }\\
operation modelling composition of localizations & for $U, V \in \LQ$ set  $U\ast V=\{A \cup B : A \in U,  B \in V, \ A\angle B \}$   &  for  $U, V \in \mathbb{U}$ set $U\ast V=\{C \cup D : C \in U, \ D \in V, \ C\angle D \}$   \\
\parbox[t]{4cm}{tread sets of \\ $\mathbb{A}$} & $T(\mathbb{A})=\{ T \in \mathbb{P} :  \exists a_i \in T \cap A_i, \ a_1\leq\dots \leq a_{n} \}$ & $T(\mathbb{A})=\{ T \in s(F) : \exists \fp_i \in T \cap A_i, \ \fp_1\supseteq \dots \supseteq\fp_{n}  \}$\\
decomposition of $T(\mathbb{A})$ & $T(\mathbb{A})=vA_1\ast vA_2 \ast \dots \ast vA_n$ &   $T(\mathbb{A})=wA_1\ast wA_2 \ast \dots \ast wA_n$\\
consequence on $\lambda_{\mathbb{A}}$ & $T(\mathbb{A})=T(\mathbb{B})$ $\Rightarrow$ $\lambda_{\mathbb{A}}\cong \lambda_\mathbb{B}$ & $T(\mathbb{A})=T(\mathbb{B})$ $\Rightarrow$ $\LL_{\mathbb{A}}\cong \LL_\mathbb{B}$ \\
\hline
\end{longtable}
\endgroup
\vspace{1mm}
\begin{theorem}
Let $\CT$ be a stratified tensor triangulated category with a noetherian Balmer spectrum and fix a finite subset $F\subseteq \emph{Spc}(\CT^c)$. Then, the iterated localizations $\LL_{\LA}$ for $\LA\in \CN(F)$ follow the results of \cite[Thm.~1.8]{itloc}, after translating the notation using the above dictionary.
\end{theorem}
\begin{proof}
It is just a matter of realizing that the proofs of all the results of \cite{itloc} translate perfectly in the new language. As we explained above, Corollary~\ref{cor-facturenoetherianbalmerspectrum-sectionstratification} is the adaptation of the fracture axiom \cite[Def.~2.1]{itloc} to this new context and it guarantees that $\LL_AX$ can be expressed as appropriate homotopy limit of compositions of single Bousfield classes, as proved in Proposition~\ref{prop-LambdaAholimchains-sectionfinitecase}.

This time we have to index the homotopy limit over the set of chains of elements in $A$, thus we pass from $\mathcal{P}(A)$ to $s(A)$. Consequently, we adapt the definition of all the posets used in \cite{itloc}: $vA$ indexing the localization $\lambda_A$, the poset $\mathbb{P}$ indexing the cubes of iterated $\phi_B$-localizations, $\mathbb{Q}$, the morphism $u\colon \mathbb{P}\rightarrow \LQ$ which lets us pass from the fully localizing cubes to the fracture diagrams, and so on.

After this passage is done, it is immediate that the proof works verbatim and all the results carry on in the new setting. For the more technical proofs working with the combinatorics of posets, we observe that even if in this situation $F$ is not a totally ordered set, its subsets we decided to consider are chains hence the previous proof really can be translated without issue.

All the core results on iterated localizations developed in \cite{itloc} are collected in \cite[Thm.~1.8]{itloc}.
%The claim of \cite[Prop.~1.10]{itloc} follows immediately from \cite[Thm.~1.8]{itloc} after unraveling the definitions.
\end{proof}

\begin{corollary}\label{cor-conjturefinitecase-sectionfinitecase}
Conjecture~\ref{conj} holds for any stratified tensor triangulated category with finite Balmer spectrum.
\end{corollary}
\begin{proof}
From \cite[Thm.~1.8]{itloc}, it follows \cite[Prop.~1.10]{itloc}, which immediately implies Conjecture~\ref{conj}.
\end{proof}

\begin{example}\label{ex-counterexformulaloclim-sectionfinite}
We provide a counterexample to Proposition~\ref{prop-LambdaAholimchains-sectionfinitecase} when $A$ is not finite.

We set $\LT=S^1$ and consider $\CT=\LT\da \CS p_{\LQ}$, the rational $\LT$-equivariant stable homotopy category. In Example~\ref{ex-GSpQstratified-sectionbalmer} we recalled that it is a stratified tensor triangulated category and we gave a description of its Balmer spectrum. Since the Lie group is abelian the conjugacy classes coincide with the closed subgroups of $\LT$. These are the trivial subgroup $e$, the cyclic subgroups $C_n$ for $n\geq 2$ and the whole group $\LT$. We have cotoral inclusions $e\leq_{ct}\LT$ and $C_n\leq_{ct}\LT$ but obviously not between different finite subgroups. Therefore, the Balmer spectrum has the following form
\begin{center}
\begin{tikzcd}
 & &  \LT \arrow[dll, leftarrow] \arrow[dl, leftarrow, shorten=1mm] \arrow[d, leftarrow] \arrow[drr, leftarrow] & & \\
 e & C_2 & C_3 & \dots & C_n &\dots
\end{tikzcd}
\end{center}
where the arrows denote the containment of Balmer primes. Its proper closed subsets consist in the finite collections of finite subgroups.

We take $A=\text{Spc}(\LT\da \CS p_{\LQ}^c)$ and show that the formula
\[ X=\LL_AX\cong \holim_{B\in s(A)'}\varphi_BX  \]
is false for a particular spectrum $X$.

Since we are considering the whole Balmer spectrum the poset $s(A)'$ has the following structure
\begin{center}
\begin{tikzcd}[column sep=0.5cm]
 & & \{\LT \} \arrow[dll] \arrow[dl] \arrow[d] \arrow[drr]& &  \\
 \{e\leq_{ct}\LT\}& \{C_2\leq_{ct}\LT\}&  \{C_3\leq_{ct}\LT\}&  \cdots & \{C_n\leq_{ct}\LT\} & \cdots \\
\{e\} \arrow[u] & \{C_2 \} \arrow[u] & \{ C_3 \} \arrow[u]& \cdots & \{ C_n \} \arrow[u] & \cdots
\end{tikzcd}
\end{center}
If $\CF$ denotes the family of finite subgroups we have $\LL_{\LT}\cong \widetilde{E}\CF\wedge -$ while for a proper subgroup $C<\LT$ we have $\LL_{C}=\Lambda_C$. Hence, $X$ should coincide with the homotopy limit of the diagram
\begin{center}
\begin{tikzcd}[column sep=0.5cm]
 & & \widetilde{E}\CF\wedge X \arrow[dll] \arrow[dl] \arrow[d] \arrow[drr]& &  \\
 \widetilde{E}\CF\wedge \Lambda_{e}X& \widetilde{E}\CF\wedge \Lambda_{C_2}X&   \widetilde{E}\CF\wedge \Lambda_{C_3}X&  \cdots &  \widetilde{E}\CF\wedge\Lambda_{C_n}X & \cdots \\
\Lambda_{e}X \arrow[u] & \Lambda_{C_2}X \arrow[u] & \Lambda_{C_3}X  \arrow[u]& \cdots & \Lambda_{C_n}X \arrow[u] & \cdots
\end{tikzcd}
\end{center}
We consider $X=E\CF_+$. A consequence of \cite[Thm.~2.2.3]{gre-rationals1equiv} is that the following decomposition holds
\[ E\CF_+\cong \bigvee_{C\in \CF} E\bc{C}. \]
For such $X$, we have $\LL_{\LT}X=0$. Thus, the homotopy limit of the previous diagram coincides with the limit of the diagram
\begin{center}
\begin{tikzcd}[column sep=0.5cm]
 0 \arrow[d]&0 \arrow[d] & 0  \arrow[d] &\cdots &0\arrow[d] & \cdots \\
 \widetilde{E}\CF\wedge \Lambda_{e}X & \widetilde{E}\CF\wedge \Lambda_{C_2}X&  \widetilde{E}\CF\wedge \Lambda_{C_3}X&  \cdots & \widetilde{E}\CF\wedge \Lambda_{C_n}X & \cdots \\
\Lambda_{e}X \arrow[u] & \Lambda_{C_2}X \arrow[u] & \Lambda_{C_3}X \arrow[u]& \cdots & \Lambda_{C_n}X \arrow[u]& \cdots
\end{tikzcd}
\end{center}
This limit consists in the product of the limits of the spans $\Lambda_C X\rightarrow \widetilde{E}\CF\wedge \Lambda_{C}X \leftarrow 0$. Forming the limit of  one of such spans coincides with taking the cofiber of the morphism $\Lambda_CX\rightarrow \widetilde{E}\CF\wedge \Lambda_{C}X$ which is $E\CF_+\wedge \Lambda_CX$. This term  can be shown to coincide with $E\bc{C}$. Thus, the canonical map from $\LL_AX$ to the homotopy limit is the usual comparison map
\[ E\CF_+\cong  \bigvee_{C\in \CF}E\bc{C}\rightarrow \prod_{C\in \CF}E\bc{C} \]
from an infinite wedge of objects to their product. 

This is not an isomorphism. Indeed, \cite[Lemma~2.3.5~(b)]{gre-rationals1equiv} provides the computation $\pi_1^{\mathbb{T}}(E\bc{C})\cong \mathbb{Q}$ for any cyclic group $C$. Therefore, as we argued in Example~\ref{ex-SpcO(2)counterexamplefractureaxiom-sectionstratification}, taking the image of this map under $[\Sigma S^0,-]$ we obtain the canonical comparison map between an infinite sum and an infinite product of copies of $\mathbb{Q}$.
\end{example}

\section{Low dimensional cases}\label{section-lowdimcases}
As one would expect, the situation when the Balmer spectrum has dimension 0 is trivial.
\begin{proposition}\label{prop-conjtruedim0-sectionlowdim}
Let $\CT$ be a stratified tensor triangulated category such that its Balmer spectrum is a noetherian space of dimension 0. Then, Conjecture~\ref{conj} holds for $\CT$.
\end{proposition}
\begin{proof}
The condition on the Balmer spectrum to be of dimension 0 means it is T1, but for spectral spaces this is equivalent to being Hausdorff (see \cite[Prop.~1.3.20]{spespa}). Noetherian Hausdorff spaces are finite discrete.

This implies that we have an equivalence of tensor triangulated categories
\begin{align*}
\CT &\simeq \prod_{\fp \in \text{Spc}(\CT^c)}L_{\fp}\CT\\
X&\mapsto (L_{\fp}X)_{\fp}.
\end{align*}
Consequently, for any subset $A\subseteq \text{Spc}(\CT^c)$ and any object $X\in \CT$ we have an isomorphism $\LL_{A}X\cong \prod_{\fp \in A}L_\fp X$. Thus, for an arbitrary tuple $\LA$ the equality $\LL_{\LA}=L_{\cap_i A_i}$ holds.

Since there are no proper inclusions of primes, all the possible chains of $\text{Spc}(\CT^c)$ reduce to singletons. Consequently, $T(\LA)=\{ \{\fp\}: \fp\in  \cap_i A_i \}$ and from this description the Conjecture follows easily.
\end{proof}

We now prove Conjecture~\ref{conj} for all Balmer spectra of dimension 1.

\begin{lemma}\label{lem-conjecturereducedthomasoncomplements-sectionlowdim}
Let $\CT$ be a stratified tensor triangulated category with noetherian Balmer spectrum. Suppose we have a finite decomposition $\emph{Spc}(\CT^c)=\bigcup_{i=1}^nZ_i$ where $Z_i$ are complements of Thomason subsets.

Suppose Conjecture~\ref{conj} holds for $L_Z\CT$ where $Z$ is any possible intersection of the above $Z_i$'s. Then, Conjecture~\ref{conj} holds for $\CT$.
\end{lemma}
\begin{proof}
We observe that by Lemma~\ref{lem-fracturecube-sectionreduction} and Remark~\ref{rmk-conditionfracturecube-sectionreduction} given an object $X\in \CT$ there exists a  cube diagram
\[  \mathcal{C}\colon \CP(\{1,\dots, n\})\rightarrow \CT \qquad I\mapsto \mathcal{C}_IX    \]
where the starting value is $\CC_{\emptyset}X=X$ and the value at the vertex labelled by $I$ is the smashing localization associated with $\bigcap_{i\in I}Z_i$, i.e. $\CC_IX=L_{\bigcap_{i\in I}Z_i}X$.

Now if we take a $k$-tuple $\LA=(A_1,\dots, A_k)$, we can apply $\LL_{\LA}$ to $\CC$ and we obtain  $\LL_{\LA}X$  expressed as the homotopy limit of the punctured cube with values $\LL_{\LA}L_{Z}$ where $Z$ ranges over all the possible intersections of the subsets $Z_i$.

By Proposition~\ref{prop-LA=LdeltaA-sectionreduction} we have that  the composite $\LL_{\LA}\LL_Z$ reduces to $\LL_{\gamma\delta(A_1,\dots, A_k, Z)}$. Since the $(k+1)$-tuple $\gamma\delta(A_1,\dots, A_k, Z)$ is downward concatenated and $Z$ is upward closed (Lemma~\ref{lem-thomason=downwardclosed-sectionbalmer}), it follows that all the elements of the tuple $\gamma\delta(A_1,\dots, A_k, Z)$ will be subsets of $Z$. Indeed, it is easy to see that $\gamma\delta(A_1,\dots, A_k, Z)=\gamma\delta(A_1\cap Z, A_2\cap Z,\dots, A_k \cap Z)$. 

We recall that for a general rigid-compactly generated tensor triangulated category $\CS$ and any complement of a Thomason subset $V\subseteq \text{Spc}(\CS^c)$, the localized category $L_V\CS$ is also a rigid-compactly generated tensor triangulated category. We saw in Remark~\ref{rmk-smahsinglocinducedinclusionbalmer-sectionbalmer} that the localization functor $\CS\rightarrow L_V\CS$ induces on the Balmer spectra the identification $V\cong \text{Spc}(L_V\CS^c)\hookrightarrow \text{Spc}(\CS^c)$. Moreover, if $\text{Spc}(\CS^c)$ is a (weakly) noetherian Balmer spectrum, so is $V\cong \text{Spc}(L_V\CS^c)$. Finally, if $\CS$ is stratified via the Balmer-Favi support, so is $L_V\CS$ by \cite[Cor.~4.9]{BHS-stratification}.

Therefore, in our situation the localized categories $L_Z\CT$ are in the form required by the assumptions of Conjecture~\ref{conj}, hence asking the Conjecture to hold in the case of these categories makes sense.

Suppose we have two tuples $\LA=(A_1,\dots, A_k)$ and $\LB=(B_1,\dots, B_l)$ such that $T(\LA)=T(\LB)$. It is immediate that this implies that for any $Z$ we have $T(A_1\cap Z,A_2\cap Z,\dots, A_k\cap Z)=T(B_1\cap Z, \dots, B_l\cap Z)$. Hence, at every vertex of the punctured cube we have a natural isomorphism between $\LL_{\LA}L_ZX$ and $\LL_{\LB}L_Z$ by the Conjecture holding on the tensor triangulated category $L_Z\CT$.

By the canonicity of these natural isomorphisms (Remark~\ref{rmk-isoconjcanonical-sectionbasics}), it follows that we obtain an isomorphism of the two punctured cubes. Thus, we have an isomorphism between their homotopy limits $\LL_{\LA}X\cong \LL_{\LB}X$ (natural in $X$).
\end{proof}

We now begin the argument for the case in dimension 1. First, we need a finiteness result.
\begin{theorem}[{\cite[Thm.~8.1.11]{spespa}}]\label{thm-spectralnoethfinitelymanyirred-sectionlowdim}
Let $Y$ be a spectral space. Then, $Y$ is noetherian if and only if for every $K\subseteq Y$ closed constructible subset $K$ has finitely many irreducible components and it does not admit an infinite chain of strictly decreasing closed constructible irreducible subsets.
\end{theorem}
\begin{corollary}\label{cor-finmanyirrednoethsectionlowdim}
Any noetherian spectral space has finitely many irreducible components.
\end{corollary}

\begin{remark}
For a Balmer spectrum an irreducible component is the closure of a maximal prime. Hence, in this case Corollary~\ref{cor-finmanyirrednoethsectionlowdim} means that a noetherian Balmer spectrum has finitely many maximal primes.
\end{remark}

\begin{lemma}\label{lemma-intersectionmaxprimesfinite-sectionlowdim}
Let $\emph{Spc}(\CT^c)$ be noetherian of dimension $1$. Take $\fp, \fq \in \emph{Spc}(\CT^c)$ two distinct Balmer primes of length $1$. Then, the intersection $[\subseteq \fp]\cap[\subseteq \fq ]=\overline{\{\fp\}}\cap\overline{\{\fq\}}$ is finite.
\end{lemma}
\begin{proof}
Observe that since $\overline{\{\fp\}}\cap\overline{\{\fq\}}\subseteq \text{Spc}(\CT^c)$ is a closed subset, it is again a noetherian spectral space. Hence, it must have finitely many irreducible components (Corollary~\ref{cor-finmanyirrednoethsectionlowdim}), which consist in the closures of some Balmer primes. But since $\fp \neq \fq$, all the elements of the intersection $\overline{\{\fp\}}\cap\overline{\{\fq\}}$ have no non-trivial inclusions. Thus, the above irreducible components are singletons and we conclude the subspace in question is finite.
\end{proof}
This will allow us to prove the Conjecture by induction. We first solve the starting step.
\begin{proposition}\label{prop-conjtruedim1irred-sectionlowdim}
Conjecture~\ref{conj} holds for a stratified tensor triangulated category $\CT$ with an irreducible noetherian Balmer spectrum of dimension 1.
\end{proposition}
\begin{proof}
The spectrum being irreducible means it has a unique maximal prime, we denote it by $\ft$.

Using Proposition~\ref{prop-LA=LdeltaA-sectionreduction} we can consider only tuples of subsets of $\text{Spc}(\CT^c)$ which are collapsed and concatenated. Thus, we can see that $\LL_{\LA}$ reduces to one of the three following forms:
\begin{enumerate}
\item $\Lambda_C$ where $\ft \not \in C$;
\item $L_{\{\ft\}\cup C}$ where $\ft\not \in C$;
\item $L_{\{\ft\}\cup C}\Lambda_D$ where $\ft \not \in C$, $\ft\not \in D$ and $C\subsetneq D$.
\end{enumerate}
If we have $T(\LA)=T(\LB)$ for two tuples, then $T(\gamma\delta\LA)=T(\LA)=T(\LB)=T(\gamma\delta\LB)$ and we can reduce to examine only tuples as in the above three cases to verify the Conjecture.

We have:
\begin{enumerate}
\item $T(C)=wC=\{ T \in s(\text{Spc}(\CT^c)) : T\cap C\neq \emptyset\}$;
\item $T(\{\ft\}\cup C)=w(\{\ft \}\cup C)=\{ T \in s(\text{Spc}(\CT^c)) : T\cap(\{\ft\}\cup C)\neq \emptyset\}$;
\item $T(\{\ft\}\cup C, D)=\{ T \in s(\text{Spc}(\CT^c)) : \exists d\in D \ \ft, d\in T \quad \vee \quad \exists c \in C \ c \in T\}$.
\end{enumerate}
If $T(\LA)=T(\LB)$ does contain $\{\ft\}$, then $\LA$ and $\LB$ must be in the form $\{\ft\}\cup C$ and $\{\ft\}\cup C'$ respectively. Since $C=\{\fq \in \text{Spc}(\CT^c) \setminus \{\ft \}: \{\fq\}\in T(\LA) \}$ and similarly $C'=\{\fq \in \text{Spc}(\CT^c) \setminus \{\ft \}: \{\fq\}\in T(\LB) \}$, we conclude $C=C'$.

If instead $T(\LA)=T(\LB)$ does not contain $\{\ft \}$, we are in the first or third case. Set $C=\{\fq \in \text{Spc}(\CT^c) \setminus \{\ft\}: \{\fq\}\in T(\LA) \}$ and $D=\{\fq \in \text{Spc}(\CT^c) \setminus \{\ft\}: \{\fq, \ft\}\in T(\LA)\}$. If $C=D$, then we must be in the first case and the description of $C$ identifies uniquely $\LA=\LB=(C)$. If $C\subsetneq D$, then we are in the third case and $\LA=\LB=(\{\ft\}\cup C, D)$.
\end{proof}
\begin{remark}\label{rmk-dim1samethreadssufficient-sectionlowdim}
Be warned that we are not stating that the collection of thread sets uniquely identifies a composition of localizations: we argued that given the collections of thread sets $T(\LA)=T(\LB)$ we can identify a unique collapsed concatenated tuple $\LC$ such that $\LL_{\LA}\cong \LL_{\LC}\cong \LL_{\LB}$. However, we cannot exclude that there might exist tuples $\mathbb{D}, \mathbb{E}$ with different thread sets $T(\mathbb{D})\neq T(\mathbb{E})$ but whose associated iterated localizations are nevertheless isomorphic $\LL_{\mathbb{D}}\cong \LL_{\mathbb{E}}$.

A priori, it is not clear why the iterated localizations as in the above cases should be distinct if they are in different forms. Indeed, the fact that $\LL_{A}g(\fp)\neq 0$ if and only if $\fp\in A$ allows us to easily differentiate cases (1) and (2) for different subsets $C$. But this is not enough to distinguish the composition of case (3) from single localizations $\LL_A$, or two of such compositions associated with different subsets $C, D$.

Even in this low dimensional case we cannot prove the inverse of Conjecture~\ref{conj} in general.

However, in concrete examples where we can carry out computations, it should be possible to prove that once reduced to the above forms the iterated localizations are distinct.

For example, consider $\CT=D(\mathbb{Z})$: we can differentiate the iterated localizations by applying them to various chain complexes and then computing the homology groups.
\end{remark}

\begin{theorem}\label{thm-conjtruedim1-sectionlowdim}
For a stratified tensor triangulated category $\CT$ with noetherian Balmer spectrum of dimension 1 Conjecture~\ref{conj} holds.
\end{theorem}
\begin{proof}
We prove the claim by induction on the number of maximal primes.

%Without loss of generality we can assume the Balmer spectrum to be connected: otherwise we decompose $\text{Spc}(\CT^c)$ in its connected components $C_1,\dots, C_n$ which are finitely many by noetherinity, 

Proposition~\ref{prop-conjtruedim1irred-sectionlowdim} deals with the starting case. Now assume the claim to be true if the Balmer spectrum has at most $n$ maximal primes.

Let us consider a tensor triangulated category $\CT$ whose Balmer spectrum $\text{Spc}(\CT^c)$ has $(n+1)$-maximal primes, say $\fp_1,\dots, \fp_{n+1}$. We define
\[U=\text{Spc}(\CT^c)\setminus \overline{\{\fp_{n+1}\}}\qquad W=\overline{\{\fp_{n+1}\}}\setminus \bigcup_{i=1}^n\overline{\{\fp_i\}} \qquad V=[\supseteq U^c\cap W^c ].  \] 
These are three complements of  Thomason subsets. The first one coincides with the Balmer primes contained in one of the first $n$ maximal primes, but not in $\fp_{n+1}$. The second is given by the Balmer primes contained exclusively in the maximal prime $\fp_{n+1}$. Finally, unraveling the definition we can write
\[ V=\{\fq\in \text{Spc}(\CT^c) : \exists i\leq n \ \exists \fp \in \overline{\{\fp_i\}}\cap \overline{\{\fp_{n+1}\}} \quad \fq\supseteq \fp \}.\]
This set consists of the generalization closure of the primes belonging to the intersection of $\overline{\{\fp_{n+1}\}}$ and $\overline{\{\fp_{i}\}}$ for some $i\leq n$.

For $L_U\CT$ the Conjecture is verified by inductive assumption, for $L_W\CT$ by Proposition~\ref{prop-conjtruedim0-sectionlowdim} or Proposition~\ref{prop-conjtruedim1irred-sectionlowdim} depending on the length of $\fp_{n+1}$. We observe that $V$ is finite by Lemma~\ref{lemma-intersectionmaxprimesfinite-sectionlowdim}. Hence, if $Z$ is one among $V, V\cap U, V\cap W, V\cap U\cap W$, the Conjecture holds for $L_{Z}\CT$ by Corollary~\ref{cor-conjturefinitecase-sectionfinitecase}. 

Therefore, Lemma~\ref{lem-conjecturereducedthomasoncomplements-sectionlowdim} allows us to conclude.
\end{proof}

\begin{example}
Recall Example~\ref{ex-DRnoethstratified-sectionbalmer}. Let $k$ be an algebraically closed field. We set $R=k[x,y]/(xy)$: this is a noetherian commutative ring of Krull dimension 1. The underlying set of its Zariski spectrum is given by the prime ideals $(x), (y)$ and $(x-\lambda, y-\mu)$ for arbitrary scalars $\lambda, \mu \in k$. We denote by these names the corresponding Balmer ideals in the Balmer spectrum $\text{Spc}(D(R)^c)$. Consequently, we have the proper inclusions $(x,y-\mu)\subset (x)$, $(x-\lambda, y)\subset (y)$ and the ideal $(x,y)$ is the only one contained in both $(x)$ and $(y)$.

We now define $A=\{(x-\lambda,y) : \lambda \in k\}$ and $B=\text{Spc}(D(R)^c)\setminus \{(x),(y)\}$, then consider $\LL_{\{(x), (y)\}}\LL_{\{(x)\}\cup A}\LL_{B}$. In this case the possible threads are
\[\begin{array}{ccccc}
(x)&=&(x) &\supset &  (x,y-\mu)\\
(x)&\supset &(x,y) &=&  (x,y) \\
(y)&\supset &(x-\lambda,y)   &=& (x-\lambda, y).
\end{array}\]
%
%\begin{alignat*}{2}
%(x)&=(x) &&\supset   (x,y-\mu)\\
%(x)&\supset (x,y) &&=   (x,y) \\
%(y)&\supset (x-\lambda,y) && = (x-\lambda, y).
%\end{alignat*}
We can easily verify that the associated thread sets are the same for $\LL_{\{(x), (y)\}}\LL_{B}$. Therefore, Theorem~\ref{thm-conjtruedim1-sectionlowdim} implies we have an isomorphism 
\[\LL_{\{(x), (y)\}}\LL_{\{(x)\}\cup A}\LL_{B}\cong \LL_{\{(x), (y)\}}\LL_{B}.\]
Actually, it is not too difficult to prove this isomorphism directly. Indeed, since $B$ is a discrete subspace and it can be split as $B=A^*\cup C\cup \{(x,y)\}$ with $A^*=A\setminus \{(x,y)\}$ and $C=B\setminus A$, we have a decomposition $\Gamma_B\mbu\cong \Gamma_{A^*}\mbu \amalg \Gamma_C\mbu\amalg \Gamma_{(x,y)}\mbu$. Hence, the corresponding equality $\LL_B=\Lambda_B\cong \Lambda_{A^*}\times \Lambda_{C}\times \Lambda_{(x,y)}$ holds. It follows that 
\[ \LL_{\{(x)\}\cup A}\LL_B\cong \LL_{\{(x)\}\cup A}(\LL_{A^*}\times \LL_{C}\times \LL_{(x,y)}  )\cong \LL_{A^*}\times \LL_{(x)}\LL_C \times \LL_{(x,y)}.\]
Using this we can compute
\begin{align*}
\LL_{\{(x), (y)\}}\LL_{\{(x)\}\cup A}\LL_{B}&\cong \LL_{\{(x), (y)\}}(\LL_{A^*}\times \LL_{(x)}\LL_C\times \LL_{(x,y)})\\
&\cong \LL_{(y)}\LL_{A^*}\times \LL_{(x)}\LL_{C}\times \LL_{\{(x), (y)\}}\LL_{(x,y)}  \cong \LL_{\{(x), (y)\}}\LL_B.
\end{align*}
\end{example}
We now prove the isomorphism of Conjecture~\ref{conj} for a special case where we consider a Balmer spectrum of dimension 2.

\begin{lemma}\label{lem-specialcesesconjdim2-sectionlowdim}
Let $\CT$ be a stratified tensor triangulated category with noetherian Balmer spectrum of dimension 2. Suppose there is a unique maximal prime $\ft$ and a unique minimal prime $\fm$. Let $A,B\subseteq \emph{Spc}(\CT^c)$ be two arbitrary collections of Balmer primes of length 1. Then, there are canonical natural isomorphisms as follows
\begin{align*}
\LL_{ \{\ft\} \cup A\cup\{\fm\}  }\LL_{\{\ft\} \cup B\cup\{\fm\}  } &\cong \LL_{\{\ft\} \cup (A\cap B)\cup\{\fm\}  }, \\
\LL_{ \{\ft\} \cup A}\LL_{\{\ft\} \cup B\cup\{\fm\}  } &\cong \LL_{ \{\ft\} \cup A}\LL_{\{\ft\} \cup (A\cap B)\cup\{\fm\}  }, \\
\LL_{ \{\ft\} \cup A \cup \{\fm\}}\LL_{B\cup\{\fm\} } &\cong \LL_{ \{\ft\} \cup (A\cap B)\cup \{\fm\}}\LL_{B\cup\{\fm\}  }.
\end{align*}
\end{lemma}
\begin{proof}
We start from the first isomorphism. We divide the set $\{\ft\} \cup B\cup\{\fm\}$ as follows: define $C=\{\ft\}\cup B\setminus A$ and $D=(A\cap B)  \cup \{\fm\}$. Clearly $\forall \fp \in D, \forall\fq \in C$ we have $\fp\not\supseteq \fq$, therefore we can apply Lemma~\ref{lem-fracturecube-sectionreduction} to obtain a homotopy pullback square as follows
\begin{center}
\begin{tikzcd}
\LL_{\{\ft\} \cup B\cup\{\fm\}}X \arrow[r] \arrow[d] & \LL_{C}X \arrow[d] \\
\LL_{D}X \arrow[r] & \LL_{C}\LL_{D}X. \arrow[ul, phantom, "\ulcorner" very near end]
\end{tikzcd}
\end{center} 
If we apply $\LL_{ \{\ft\} \cup A\cup\{\fm\}}$ to the above diagram we still obtain an homotopy pullback. We observe $\gamma\delta( \{\ft\} \cup A\cup\{\fm\},C)=(\{\ft\})$ and $\gamma\delta( \{\ft\} \cup A\cup\{\fm\}, D)=(D)$, therefore by Proposition~\ref{prop-LA=LdeltaA-sectionreduction} the resulting square is in the form
\begin{center}
\begin{tikzcd}
\LL_{ \{\ft\} \cup A\cup\{\fm\} }\LL_{\{\ft\} \cup B\cup\{\fm\}}X \arrow[r] \arrow[d] & \LL_{\ft}X \arrow[d] \\
\LL_{(A\cap B)\cup\{\fm\}}X \arrow[r] & \LL_{\ft}\LL_{D}X. \arrow[ul, phantom, "\ulcorner" near end]
\end{tikzcd}
\end{center} 
We notice the punctured square is the same we would obtain for $\{\ft\}\cup (A\cap B)\cup \{\fm\}$ invoking Lemma~\ref{lem-fracturecube-sectionreduction} and using the decomposition in the two subsets $\{\ft\}$ and $(A\cap B)\cup\{\fm\}$. This proves the claimed isomorphism.

Now on to the second. As before, let us consider again the fracture square
\begin{center}
\begin{tikzcd}
\LL_{\{\ft\} \cup B\cup\{\fm\}}X \arrow[r] \arrow[d] & \LL_{C}X \arrow[d] \\
\LL_{D}X \arrow[r] & \LL_{C}\LL_{D}X \arrow[ul, phantom, "\ulcorner" very near end]
\end{tikzcd}
\end{center}
but now we apply $\LL_{\{\ft\}\cup A}$ to this. Using the reduction of Proposition~\ref{prop-LA=LdeltaA-sectionreduction}, we obtain the homotopy pullback
\begin{center}
\begin{tikzcd}
\LL_{\{\ft\}\cup A}\LL_{\{\ft\} \cup B\cup\{\fm\}}X \arrow[r] \arrow[d] & \LL_{\ft}X \arrow[d] \\
\LL_{\{\ft\}\cup A}\LL_{D}X \arrow[r] & \LL_{\ft}\LL_{D}X. \arrow[ul, phantom, "\ulcorner"  near end]
\end{tikzcd}
\end{center} 
As we mentioned above, the decomposition of $\{\ft\}\cup(A\cap B)\cup \{\fm\}$ in the two sets $\{\ft\}$ and $D$ gives the homotopy pullback
 \begin{center}
\begin{tikzcd}
\LL_{\{\ft\} \cup (A\cap B)\cup\{\fm\}}X \arrow[r] \arrow[d] & \LL_{\ft}X \arrow[d] \\
\LL_{D}X \arrow[r] & \LL_{\ft}\LL_{D}X \arrow[ul, phantom, "\ulcorner" very near end]
\end{tikzcd}
\end{center}
and applying $\LL_{\{\ft\}\cup A}$ to it we obtain the same span in the lower right corner as in the square before this. Therefore, since  both compositions $\LL_{ \{\ft\} \cup A}\LL_{\{\ft\} \cup B\cup\{\fm\}  }$ and $\LL_{ \{\ft\} \cup A}\LL_{\{\ft\} \cup (A\cap B)\cup\{\fm\}  }$ can be expressed as the homotopy limit of the same diagram they must coincide.

We finally conclude with the third isomorphism. This time we use the decomposition $\{\ft\}\cup A\cup \{\fm\}$ as the union of $\{\ft\}\cup (A\cap B)$ and $(A\setminus B)\cup \{\fm\}$ to get the homotopy pullback square
\begin{center}
\begin{tikzcd}
\LL_{\{\ft\} \cup A\cup\{\fm\}}X \arrow[r] \arrow[d] & \LL_{(A\setminus B)\cup\{\fm\}}X \arrow[d] \\
\LL_{\{\ft\}\cup (A\cap B)}X \arrow[r] & \LL_{\{\ft\}\cup (A\cap B)}\LL_{(A\setminus B)\cup\{\fm\}}X. \arrow[ul, phantom, "\ulcorner" very near end]
\end{tikzcd}
\end{center}
If we consider this for $X=\LL_{B\cup \{\fm\}}X'$ and use Proposition~\ref{prop-LA=LdeltaA-sectionreduction} we obtain the square  
\begin{center}
\begin{tikzcd}
\LL_{\{\ft\} \cup A\cup\{\fm\}}\LL_{B\cup \{\fm\}}X' \arrow[r] \arrow[d] & \LL_{\fm}X' \arrow[d] \\
\LL_{\{\ft\}\cup (A\cap B)}\LL_{B\cup \{\fm\}}X' \arrow[r] & \LL_{\{\ft\}\cup (A\cap B)}\LL_{\fm}X'. \arrow[ul, phantom, "\ulcorner" very near end]
\end{tikzcd}
\end{center}
But the decomposition of $\{\ft\}\cup (A\cap B)\cup \{\fm\}$ in $\{\ft\}\cup (A\cap B)$ and $\{\fm\}$ induces the homotopy pullback
\begin{center}
\begin{tikzcd}
\LL_{\{\ft\} \cup (A\cap B)\cup\{\fm\}}X \arrow[r] \arrow[d] & \LL_{\fm}X \arrow[d] \\
\LL_{\{\ft\}\cup (A\cap B)}X \arrow[r] & \LL_{\{\ft\}\cup (A\cap B)}\LL_{\fm}X. \arrow[ul, phantom, "\ulcorner" very near end]
\end{tikzcd}
\end{center}
If we consider this for  $X=\LL_{B\cup \{\fm\}}X'$ we obtain the same punctured square as in the diagram before of this, expressing $\LL_{\{\ft\} \cup A\cup\{\fm\}}\LL_{B\cup \{\fm\}}X'$ as homotopy pullback. Therefore, we conclude there is an isomorphism $\LL_{\{\ft\} \cup A\cup\{\fm\}}\LL_{B\cup \{\fm\}}X'\cong \LL_{\{\ft\} \cup (A\cap B)\cup\{\fm\}}\LL_{B\cup \{\fm\}}X'$.
\end{proof}

\begin{remark}
Observe that we have the corresponding equalities of collections of thread sets 
\begin{align*}
T(\{\ft\} \cup A\cup\{\fm\} , \{\ft\} \cup B\cup\{\fm\} )&=T(\{\ft\} \cup (A\cap B)\cup\{\fm\} ),\\
 T(\{\ft\}\cup A, \{\ft\}\cup B\cup \{\fm\}  )&=T(\{\ft\}\cup A, \{\ft\}\cup (A\cap B)\cup \{\fm\} ),\\
T(\{\ft\}\cup A\cup \{\fm\}, B\cup \{\fm\}  )&=T(\{\ft\}\cup (A\cap B)\cup \{\fm\}, B\cup \{\fm\} ).
\end{align*}
Thus, Lemma~\ref{lem-specialcesesconjdim2-sectionlowdim} presents particular instances of the isomorphism of Conjecture~\ref{conj}.
\end{remark}

\begin{corollary}\label{cor-LtALtAcapBmLBm=LtALBm-sectionlowdim}
Let $\CT$ be a stratified tensor triangulated category with noetherian Balmer spectrum of dimension 2. Suppose there is a unique maximal prime $\ft$ and a unique minimal prime $\fm$. Let $A,B\subseteq \emph{Spc}(\CT^c)$ be two arbitrary collections of Balmer primes of length 1. Then, there is a canonical natural isomorphism \[ \LL_{\{\ft\}\cup A}\LL_{\{\ft\}\cup (A\cap B)\cup \{\fm\}}\LL_{B\cup \{\fm\}}\cong  \LL_{\{\ft\}\cup A}\LL_{B\cup \{\fm\}}. \]
\end{corollary}
\begin{proof}
We denote by $S$ the set $\text{Spc}(\CT^c)\setminus \{\fm, \ft\}$. Then, concatenating the isomorphisms of Lemma~\ref{lem-specialcesesconjdim2-sectionlowdim} we obtain
\begin{align*}
\LL_{\{\ft\}\cup A}\LL_{\{\ft\}\cup (A\cap B)\cup \{\fm\}}\LL_{B\cup \{\fm\}}&\cong \LL_{\{\ft\}\cup A}\LL_{\{\ft\}\cup (A\cap B)\cup (S\setminus A)\cup \{\fm\}}\LL_{B\cup \{\fm\}}\\
&\cong \LL_{\{\ft\}\cup A}\LL_{\{\ft\}\cup (A\cap B)\cup(S\setminus A)\cup (S\setminus B)\cup \{\fm\}}\LL_{B\cup \{\fm\}}\\
&\cong \LL_{\{\ft\}\cup A}\LL_{B\cup \{\fm\}}
\end{align*}
where the last isomorphism comes from $(A\cap B)\cup (S\setminus A)\cup (S\setminus B)=S$.
\end{proof}

\begin{proposition}\label{prop-conjtruedim2irredonemin-sectionlowdim}
Let $\CT$ be a stratified tensor triangulated category with noetherian Balmer spectrum of dimension 2. Suppose there is a unique maximal prime $\ft$ and a unique minimal prime $\fm$. Then, Conjecture~\ref{conj} holds for $\CT$.
\end{proposition}
\begin{proof}
Using Proposition~\ref{prop-LA=LdeltaA-sectionreduction}, Lemma~\ref{lem-specialcesesconjdim2-sectionlowdim} and Corollary~\ref{cor-LtALtAcapBmLBm=LtALBm-sectionlowdim}, we can reduce all the possible compositions of localizations to one of the following forms, where $A^1, B^1, C^1$ denote collections of Balmer primes of length $1$.
\begin{enumerate}
\item $\LL_{A^1}$;
\item $\LL_{\{\ft\} \cup A^1}$;
\item $\LL_{A^1\cup \{\fm\}}$;
\item $\LL_{\{\ft\}\cup A^1 \cup \{\fm\}}$;
\item $\LL_{\{\ft\} \cup A^1}\LL_{B^1}$ with $A^1\subsetneq B^1$;
\item $\LL_{A^1}\LL_{B^1\cup \{\fm\}}$ with $B^1\subsetneq A^1$;
\item $\LL_{\{\ft\}\cup A^1}\LL_{B^1\cup \{\fm\}}$;
\item $\LL_{\{\ft\}\cup A^1}\LL_{\{\ft\}\cup B^1 \cup \{\fm\}}$ with $B^1\subsetneq A^1$;
\item $\LL_{\{\ft\}\cup A^1 \cup \{\fm\}}\LL_{B^1\cup \{\fm\}}$ with $A^1\subsetneq B^1$;
\item $\LL_{\{\ft\}\cup A^1}\LL_{B^1}\LL_{C^1\cup \{\fm\}}$ with $A^1,C^1\subsetneq B^1$;
\item $\LL_{\{\ft\}\cup A^1}\LL_{\{\ft\}\cup B^1 \cup \{\fm\}}\LL_{C^1\cup \{\fm\}}$ with $B^1\subsetneq A^1\cap C^1$.
\end{enumerate}
We now prove that a collection of thread sets for one of the tuples as above, let us denote it $T$, uniquely individuates the iterated localization associated with it.

Case (1) can be characterized by these properties:
\begin{itemize}
\item $\{\ft\}, \{\fm\},\{\ft, \fm\}$ are not thread sets;
\item we have the following equalities of sets
\[  \{\fp \in \text{Spc}(\CT^c): \fp\neq \ft, \fm \quad \{\fp, \ft\}\in T \}=\{\fp \in \text{Spc}(\CT^c): \fp\neq \ft, \fm \quad \{\fp\}\in T \}  \]
and 
\[  \{\fp \in \text{Spc}(\CT^c): \fp\neq \ft, \fm \quad \{\fp, \fm\}\in T \}=\{\fp \in \text{Spc}(\CT^c): \fp\neq \ft, \fm \quad \{\fp\}\in T \},  \]
which exclude cases (5) and (6) respectively;
\item  finally, to distinguish it from case (10) when $A^1=C^1$, we have the equalities
\begin{align*}
\{  \fp \in \text{Spc}(\CT^c) : \fp \neq \fm, \ft \quad \{\fp, \ft\} \in T \}&= \{  \fp \in \text{Spc}(\CT^c) : \fp \neq \fm, \ft \quad \{\fp, \fm\} \in T \} \\
&=\{  \fp \in \text{Spc}(\CT^c) : \fp \neq \fm, \ft \quad \{\fp, \ft, \fm\} \in T \}. 
\end{align*}
\end{itemize}
We can individuate uniquely $A^1$ as the set of Balmer primes of length $1$ in the last equality, or more simply as $\{ \fp \in \text{Spc}(\CT^c) : \fp \neq \fm, \ft \quad \{\fp\}\in T \}$.

Cases (2) and (8) are the only ones such that $\{\ft\}$ is a thread set, while $\{\fm\}$ is not. At this point (2) can be differentiated from (8) by the equality
\[ \{\fp\in \text{Spc}(\CT^c): \fp\neq \ft, \fm \quad \{\fp\}\in T \}=\{\fp\in \text{Spc}(\CT^c): \fp\neq \ft, \fm \quad \{\fp, \fm\}\in T \} \]
and this set uniquely identifies $A^1$. For case (8) instead we have a proper inclusion 
\[ \{\fp\in \text{Spc}(\CT^c): \fp\neq \ft, \fm \quad \{\fp\}\in T \}\subsetneq\{\fp\in \text{Spc}(\CT^c): \fp\neq \ft, \fm \quad \{\fp, \fm\}\in T \}\]
and these two subsets recover $B^1$ and $A^1$ respectively.

Case (3) and (9) can be treated as (2) and (8) by inverting the role of $\ft$ and $\fm$.

Case (4) is the only one admitting as thread sets both $\{\ft\}$ and $\{\fm\}$. Here we can reconstruct $A^1$ with no ambiguity as $\{ \fp \in \text{Spc}(\CT^c) : \fp \neq \fm, \ft \quad \{\fp\}\in T \}$.

Case (5) is characterized by the properties
\begin{itemize}
\item $\{\ft, \fm\}\not \in T$;
\item the equality 
\[ \{\fp\in \text{Spc}(\CT^c) : \fp\neq \ft, \fm \quad \{\fp\}\in T\}=\{\fp\in \text{Spc}(\CT^c) : \fp\neq \ft, \fm \quad \{\fp,\fm\}\in T\}  \] 
excludes case (6);
\item if we set $A=\{ \fp \in \text{Spc}(\CT^c) : \fp \neq \fm, \ft \quad \{\fp\}\in T \}$ and $B=\{ \fp \in \text{Spc}(\CT^c) : \fp \neq \fm, \ft \quad \{\fp, \ft \}\in T \}$, we have a proper inclusion $A \subsetneq B$, so we can exclude case (1) and case (10) when $A^1\cap C^1=C^1$;
\item we have the equality
\[ \{  \fp \in \text{Spc}(\CT^c) : \fp \neq \fm, \ft \quad \{\fp, \ft\} \in T \} =\{  \fp \in \text{Spc}(\CT^c) : \fp \neq \fm, \ft \quad \{\fp, \ft, \fm\} \in T \} \]
to distinguish it from the case (10) when $A^1\cap C^1 \subsetneq C^1$; 
\end{itemize}
Furthermore, the above sets $A, B$ individuate respectively the sets $A^1, B^1$ without ambiguity.

Case (6) can be treated as case (5) after inverting the role of $\ft$ and $\fm$.

Case (7) and (11) are the only ones with $\{\fm, \ft\}\in T$ but $\{\ft\}, \{\fm\}\not \in T$. We can differentiate them by the following fact: we define the sets
\begin{align*}
 D&=\{\fp\in \text{Spc}(\CT^c): \fp\neq \ft,\fm \quad \{\fm,\fp\}\in T\} \\
 E&=\{\fp\in \text{Spc}(\CT^c): \fp\neq \ft,\fm \quad \{\ft,\fp\}\in T\} \\
 F&=\{\fp\in \text{Spc}(\CT^c): \fp\neq \ft,\fm \quad \{\fp\}\in T\}.
\end{align*}
Then, in case (7) we must have $D\cap E=F$, while in case (11) the proper inclusion $F\subsetneq D\cap E$ holds.
In case (7) the sets $D$ and $E$ identify uniquely the sets $A^1$ and $B^1$ respectively. While for case (11) $D$ recovers $A^1$, $F$ recovers $B^1$ and finally $E$ recovers $C^1$.

Finally, we can characterize case (10) by the facts that $\{\ft, \fm\}\not \in T$ and that if we define $B=\{\fp \in \text{Spc}(\CT^c) : \fp\neq \fm, \fp \quad \{\fp, \fm,\ft\}\in T\}$, $A=\{ \fp \in \text{Spc}(\CT^c) : \fp \neq \fm, \ft \quad \{\fp, \fm \}\in T \}$ and $C=\{ \fp \in \text{Spc}(\CT^c) : \fp \neq \fm, \ft \quad \{\fp, \ft \}\in T \}$ we have proper inclusions $A\subsetneq B$ and $C\subsetneq B$, which distinguish this case from (1), (5) and (6). 
\end{proof}

\begin{theorem}\label{thm-conjtruedim2irred-sectionlowdim}
Let $\CT$ be a stratified tensor triangulated category with noetherian Balmer spectrum of dimension 2. Suppose the Balmer spectrum has finitely many minimal primes. Then, Conjecture~\ref{conj} holds true for $\CT$.
\end{theorem}
\begin{proof}
By Corollary~\ref{cor-finmanyirrednoethsectionlowdim}, the Balmer spectrum  has finitely many primes of length 2, we denote them by $\ft_1, \ft_2, \dots, \ft_n$. We argue by induction on $n$.

Let us begin with the starting case $n=1$. If $\fm_1,\dots, \fm_k$ are the minimal Balmer primes, we define $Z_i=[\supseteq \fm_i]$ for $1\leq i \leq k$ to be the cofamilies generated by them. We observe that for $L_{Z_i}\CT$ the conjecture holds by Theorem~\ref{thm-conjtruedim1-sectionlowdim} and Proposition~\ref{prop-conjtruedim2irredonemin-sectionlowdim}. If instead $Z$ is the intersection of two or more or the above $Z_i$'s, then it must have dimension $0$ or $1$. Thus, the conjecture holds for $L_Z\CT$ by Proposition~\ref{prop-conjtruedim0-sectionlowdim} or Theorem~\ref{thm-conjtruedim1-sectionlowdim}.

Therefore, Lemma~\ref{lem-conjecturereducedthomasoncomplements-sectionlowdim} shows the conjecture is true in the starting case.

Now we deal with the induction step. Suppose $n>1$, then for any $1 \leq i \leq n$ we can define the subset $U_i=[\not \subseteq \ft_i]$ of $\text{Spc}(\CT^c)$ which is clearly upward closed, hence the complement of a Thomason subset by noetherianity. Each $U_i$ is a spectral space of dimension 2 with finitely many minimal elements and with a number of length 2 elements strictly less than $n$. Thus, by inductive assumption the conjecture holds for $L_{U_i}\CT$. If $U$ is the intersection of two or more of these $U_i$'s, then either it has dimension 2 but less than $n$ elements with length 2 or it has dimension strictly less than 2. In the former case the conjecture for $L_U\CT$ holds again by inductive assumption, in the latter case it holds by Proposition~\ref{prop-conjtruedim0-sectionlowdim} and Theorem~\ref{thm-conjtruedim1-sectionlowdim}. Therefore, we can use Lemma~\ref{lem-conjecturereducedthomasoncomplements-sectionlowdim} to conclude.
\end{proof}

\begin{example}
Let us consider $\CT=\mathbb{T}^2 \da \CS p_{\LQ}$, the rational stable homoropy category equivariant with respect to the torus $\mathbb{T}^2=S^1\times S^1$. We saw in Example~\ref{ex-GSpQstratified-sectionbalmer} that this is a stratified category, whose Balmer spectrum is homeomorphic to $\Gamma \mathbb{T}^2$ with the $zf$-topology.

By \cite[Rmk.~6.2]{gre-rationalbalmer} and \cite[Lemma~6.3]{gre-rationalbalmer} all primes of $\text{Spc}(\mathbb{T}^2 \da \CS p_{\LQ}^c)$ are visible, thus this is a noetherian space. 

A priori $\Gamma \mathbb{T}^2$ has infinitely many minimal points, corresponding to the finite subgroups of $\mathbb{T}^2$. However, if we limit ourselves to consider only finitely many of these finite subgroups and the closed subgroups of $\mathbb{T}^2$ cotorally containing them, we can apply Theorem~\ref{thm-conjtruedim2irred-sectionlowdim}.

There exist countably many subgroups isomorphic to the circle $S^1$, we label these by $S_i$ for $i\in \mathbb{N}$. Clearly, if $e$ denotes the trivial subgroup, we have $S_i\geq_{ct}e$. For any $i$ we can find a finite subgroup $F_i$ such that $S_i\geq_{ct}F_i$ and $S_i\not \geq_{ct}F_j$ for $i\neq j$.

Let us fix a natural number $n$. We define the following subsets of the Balmer spectrum
\begin{align*}
A_0&=\{S_i : i\in \mathbb{N}\} \\
A_j&=\{S_i : i\geq j  \} \cup \{F_i : i<j\} \cup \{e\}\qquad 1\leq j \leq n \\
A_{n+1}&=\{F_i : 0\leq i \leq n\} \cup \{e\}.
\end{align*}
Then, the tuple $\LA=(A_0, A_1,\dots, A_{n+1})$ is collapsed and concatenated, but we have a further reduction
\[ \LL_{\LA}\cong \LL_{A_0}\LL_{A_{n+1}} \]
since $T(\LA)=T(A_0, A_{n+1})$.
\end{example}

\end{document}